\documentclass[11pt]{amsart}

\usepackage{amssymb}
\usepackage{amsthm}
\usepackage{amsmath}
\usepackage{graphicx}
\usepackage{amsaddr}
\usepackage{comment}
\usepackage[usenames,dvipsnames]{xcolor}
\usepackage[margin = 1in]{geometry}
\usepackage{enumerate}
\usepackage[toc,page]{appendix}
\usepackage{lineno}
\usepackage{color}
\usepackage{bbm}
\usepackage[colorlinks=true, allcolors=blue]{hyperref}
\usepackage[version=4]{mhchem}
\usepackage{siunitx}

\usepackage[normalem]{ulem}

\definecolor{mygreen}{rgb}{0.1,0.75,0.2}

 \newtheorem{thm}{Theorem}[section]
 \newtheorem{cor}[thm]{Corollary}
 \newtheorem{lem}[thm]{Lemma}
 \newtheorem{prop}[thm]{Proposition}
  \newtheorem*{thm*}{Main Theorem} 

 \newtheorem{defn}[thm]{Definition}
 
 \newtheorem{asmp}[thm]{Assumption}

 \newtheorem{rem}[thm]{Remark}

 \numberwithin{equation}{section}

\DeclareMathOperator{\BUC}{BUC}

\providecommand{\bbs}[1]{\left(#1\right)}

\newcommand{\pt}{\partial}
\newcommand{\eps}{\varepsilon}
\newcommand{\ud}{\,\mathrm{d}}
\newcommand{\8}{\infty}

\newcommand{\bR}{\mathbb{R}}
\newcommand{\bZ}{\mathbb{Z}}

\newcommand{\cal}{\mathcal}

\newcommand{\cv}{X^{\scriptscriptstyle{\text{h}}}} 
\newcommand{\vx}{\vec{x}}

\newcommand{\N}{\mathbb{N}}
\newcommand{\R}{\mathbb{R}}

\newcommand{\cE}{\mathcal{E}}

\newcommand{\cC}{\mathcal{C}}

\newcommand{\e}{\mathrm{e}}

\renewcommand{\L}{\mathrm{L}}
\newcommand{\W}{\mathrm{W}}

\newcommand{\D}{\mathrm{D}}
\renewcommand{\d}{\mathrm{d}}

\definecolor{rot}{rgb}{1.000,0.000,0.000}

\newcommand{\eff}{\mathrm{eff}}
\newcommand{\fast}{\mathrm{fast}}
\newcommand{\slow}{\mathrm{slow}}

\newcommand{\MS}{{\cal M}_{\S}}

\renewcommand{\S}{\mathrm{S}}

\newcommand{\sQ}{\mathsf{Q}}
\newcommand{\sC}{\mathsf{C}}
\newcommand{\sR}{\mathsf{R}}
\newcommand{\sH}{\mathsf{H}}
\newcommand{\sL}{\mathsf{L}}
\newcommand{\sJ}{\mathsf{J}}
\newcommand{\sv}{\mathsf{v}}
\newcommand{\sq}{\mathsf{q}}
\renewcommand{\sp}{\mathsf{p}}
\newcommand{\su}{\mathsf{u}}

\newcommand{\WM}{\mathbb{G}}

\newcommand{\dom}{\sC}

\newcommand{\x}{\mathrm{x}}

\begin{document}

\title[Convergence of Fast-Slow Chemical Reactions]{Fast-slow chemical reactions:  convergence of Hamilton-Jacobi equation and variational representation
}

\author[Y. Gao]{Yuan Gao}
\address{Department of Mathematics, Purdue University, West Lafayette, IN, 47906,  USA}
\email{gao662@purdue.edu}
\author[A. Stephan]{Artur Stephan}
\address{Vienna University of Technology, Institute of Analysis and Scientific Computing,\\ Wiedner Hauptstr. 8-10, A-1040 Wien, Austria}
\email{artur.stephan@tuwien.ac.at}

\date{\today}

\maketitle

\begin{abstract}
    Microscopic behaviors of chemical reactions  can be described by a random time-changed Poisson process, whose   large-volume limit   determines the  macroscopic behaviors of species concentrations, including both typical and non-typical trajectories. When the reaction intensities (or fluxes) exhibit a separation of fast-slow scales, the macroscopic typical trajectory is governed by a system of $\eps$-dependent nonlinear reaction rate equations (RRE), while the non-typical trajectories deviating from the typical ones are characterized by an $\eps$-dependent exponentially nonlinear Hamilton-Jacobi equation (HJE). 
 In this paper, for general chemical reactions, we study the fast-slow limit as $\eps\to 0$ for the viscosity solutions of the associated HJE with a state-constrained boundary condition. We identify the limiting effective HJE on a slow manifold, along with an effective variational representation for the solution. Through the uniform convergence of the viscosity solutions and the $\Gamma$-convergence of the variational solution representations, we rigorously show that all   non-typical (and also typical) trajectories   are concentrated  on the slow manifold and the effective macroscopic dynamics are described by the coarse-grained RRE and HJE, respectively.
This approach for studying the fast-slow limit is applicable  to, but  not limited to, reversible chemical reactions described by gradient flows.

\end{abstract}

\section{Introduction}
In practical biochemical reactions, the separation of time scales among different reactions is a common phenomenon. For instance,  a gene in its active state transcribes mRNA at a much faster rate than when it is inactive \cite{li2011central, PhysRevLett}.  An important modeling question is to find an accurate and effective description of the system on the slow time scale such that the essential features, including macroscopic typical and non-typical behaviors,  are still present. The mathematical question then is to show that both the multiscale and the effective descriptions are close to each other in the practical limit.  

In this paper, we are interested not only in the macroscopic typical behavior of chemical reaction systems, but rather the non-typical fluctuations deviating from the typical dynamics. To be more precise,  various macroscopic behaviors of the chemical reactions result from a multiscale stochastic random time-changed Poisson representation (see reaction process $\cv$ in \eqref{CR}). In the law of large numbers regime, the typical trajectory in terms of the concentration of species is described by a nonlinear ODE for the concentrations  $x^\eps_i$ of species $i\in \mathcal{I}$, known as the \textit{reaction-rate equation} (RRE) 
\begin{equation} \label{eq:FastSlowRRE}
\begin{aligned}
     \frac{\ud x^\eps}{\ud t}=&\sum_{r\in{\cal R}}\left(\Psi_{r,\eps}^{+}(x)  -\Psi_{r,\eps}^{-}(x)\right)\gamma_{r}\\
     =&\sum_{r\in{\cal R}_{\slow}}\left(\Psi_{r}^{+}(x^\eps)-\Psi_{r}^{-}(x^\eps)\right)\gamma_{r}+\frac{1}{\eps}\sum_{r\in{\cal R}_{\fast}}\left(\Psi_{r}^{+}(x^\eps)-\Psi_{r}^{-}(x^\eps)\right)\gamma_{r}, 
     \end{aligned} \tag{{\sf RRE$_{\eps}$}}
\end{equation}
where $\Psi_{r}^{\pm}$ represents the forward/backward intensity functions and $\gamma_r$ is the reaction vector for the $r$-th fast or slow reaction with $r\in{\cal R}_{\slow}$ or $r\in{\cal R}_{\fast}$, respectively. Here, we assume that only two time scales are present, and that  the fast intensity functions are all on the   $O(1/\eps)$ scale. 
In the large deviation regime, the non-typical trajectory can be 
characterized by a Hamilton-Jacobi equation (HJE) with its variational solution representation: given initial data $u_0^\eps\in C^1(\Omega),$ 
\begin{equation}\label{eq:FastSlowHJE}
\begin{aligned}
    \pt_t u^\eps(x,t) +  H_\eps(x, \nabla u^\eps(x,t)) \leq 0, \quad (x,t)\in \Omega \times (0,+\8);\\
    \pt_t u^\eps(x,t) +  H_\eps(x, \nabla u^\eps(x,t)) \geq 0, \quad (x,t)\in \overline{\Omega} \times (0,+\8);\\
    u^\eps(x,t)=\inf_{\mathrm{x}\in \mathrm{AC}([0,t]; \overline{\Omega}),\, \mathrm{x}(t)=x} \big(u^\eps_0(\mathrm{x}(0)) +  
   \int_0^t L_\eps(\mathrm{x}_s,\dot{\mathrm{x}}_s)\ud s\big).
   \end{aligned}\tag{{\sf HJE$_{\eps}$}}
\end{equation}
This HJE results from the nonlinear semigroup solution in the WKB expansion (cf. \cite{Kubo73, Hu87, GL22}) of the chemical master equation (see Section \ref{ss:LDP} for details).  
 The multiscale Hamiltonian $H_\eps$ and Lagrangian $L_\eps$ are related by the Legendre transform, and
\begin{equation}\label{eq:Heps}
\begin{aligned}
H_{\eps}(x,p) 
  :=&\sum_{r\in{\cal R}_{\slow}}\Psi_{r}^{+}(x)\left(\e^{\gamma_{r}\cdot p}-1\right)+\Psi_{r}^{-}(x)\left(\e^{-\gamma_{r}\cdot p}-1\right)\\
  &+\frac{1}{\eps}\sum_{r\in{\cal R}_{\fast}}\Psi_{r}^{+}(x)\left(\e^{\gamma_{r}\cdot p}-1\right)+\Psi_{r}^{-}(x)\left(\e^{-\gamma_{r}\cdot p}-1\right).
\end{aligned} 
\end{equation}
   Here $\Omega\subset \R^I$ is the domain of concentration vectors, which we assume to be bounded and having   a positive distance to the boundary $\partial \R^I_+$.   
We refer to Section \ref{sec5} for more details, where viscosity solutions of state constrained HJE on bounded domains $\Omega$ are defined in Definition \ref{def:vis}. 

In the above descriptions, $\eps$ is fixed. To derive the effective descriptions, we analyze the limit  
$\eps\to 0$ for   \eqref{eq:FastSlowRRE} and \eqref{eq:FastSlowHJE}.

For the typical behaviors described by \eqref{eq:FastSlowRRE} with multiscale reaction intensities, \cite{bothe2003reaction} obtained the effective reduced RRE using   methods from deterministic dynamical system theory for singularly perturbed systems.   {The    set of fast equilibria, ${\cE}_\fast=\left\{ x:  \,\Psi_{r}^{+}(x)=\Psi_{r}^{-}(x),\, \forall r\in{\cal R}_{\fast}\right\}$} was determined in \cite{bothe2003reaction}. This set later defines the slow manifold $\MS=\cE_\fast\cap \overline{\Omega}$, to which solutions of   \eqref{eq:FastSlowRRE} converge.    On the slow manifold the effective dynamics can either be described by projected coordinates, or by coarse-grained variables that parametrize the slow manifold. \cite{kang2013separation} established a systematic probabilistic approach to determine effective chemical master equations and RRE applicable to large chemical networks. Under the detailed balance assumption, i.e., assuming the existence of an invariant measure such that the reaction process corresponding to \eqref{eq:FastSlowRRE} is reversible, gradient flow approach is used to study the convergence of \eqref{eq:FastSlowRRE} to the effective dynamics on the slow manifold;  see   \cite{MielkeStephan} for linear reaction systems and see \cite{MielkePeletierStephan} for nonlinear reversible reaction systems.

In this paper, for general non-equilibrium chemical reactions,  we  study the convergence of the corresponding \eqref{eq:FastSlowHJE} with state constraints and the characterization of its solution representation in the limit $\eps\to 0$. In particular, we are going to show that solutions  $u^\eps$ of \eqref{eq:FastSlowHJE}  converge to $u^*$, where the limit $u^*$ solves a corresponding HJE involving an effective Hamiltonian $H_\eff$ and an effective variational representation.   This  tracks and identifies the limiting behavior of the large deviation  rate function for non-typical paths as $\eps \to 0$. The effective Hamiltonian dynamics for non-typical trajectories also recovers the effective RRE. 

   \textbf{Challenges and strategies:}
  The challenges are brought by (i) the large deviation from the typical path described by RRE, (ii) the irreversibility of general reactions,  and (iii) the singularity in both the variational representation and HJE due to fast-scale reactions.  
First, to capture the multiscale dynamics for non-typical path, we study directly \eqref{eq:FastSlowHJE}. The typical trajectories solving \eqref{eq:FastSlowRRE} or the effective RRE can then be recovered as a special bi-characteristic for the  \eqref{eq:FastSlowHJE} or the effective HJE.   In contrast to solutions of the RRE, where the boundary is repelling,   we have to treat non-typical paths near the boundary carefully. In order to identify the effective variational functional in the asymptotic limit, we use $\Gamma$-convergence techniques.  The action functional provides a priori estimates even when the trajectories are close to the boundary. Hence, our result on    the $\Gamma$-convergence  holds on $\bR^I_{\geq 0}$. However, to rule out degeneracy in the Hamiltonian due to  loss of coercivity  near the boundary, we restrict  the analysis of HJE solutions to a suitable working domain and consider a state-constrained HJE.  
Secondly, for the irreversible issue, one loses the gradient flow structures in RRE, so the   convergence to the effective dynamics on the slow manifold can not be obtained via the limiting behavior of gradient flows. We propose a weaker fast detailed balance condition (see \eqref{FDB}) to ensure the fast part of RRE being a gradient flow. It requires the existence of a $x_{s}^{*}\in\bR^I_+:= \{x\in \bR^I;\, x_i > 0 \text{ for all } i\}$ such that $\Psi_{r}^{+}(x_{s}^{*})=\Psi_{r}^{-}(x_{s}^{*})$ holds for all $r\in{\cal R}_{\fast}.$ {   Following \cite{MielkePeletierStephan}, we assume in addition that the fast part of the RRE  satisfies the so-called \textit{Unique fast equilibrium condition} \eqref{eq:UFEC} to exclude repelling boundary equilibria.  These two conditions are to ensure the compactness of non-typical paths beyond the reversibility  and also allow    linking the  effective dynamics on the slow manifold to projected or coarse-grained dynamics via a uniquely determined analytic reconstruction map $\sR$.}  Relying on this compactness and an explicit reconstruction of the reactive fluxes for the fast-slow system from the effective action functional $\mathcal{L}_\eff$, we are able to show $\Gamma$-convergence of the action functionals $\mathcal{L}_\eps$ (see Theorem \ref{thm:GammaConvergence}).  Finally, due to the singularity in the coefficients of  Hamiltonian $H_\eps$, the variational representation for the state-constraint solution is not uniform in $\eps$. We take the singular limit in both the viscosity solutions to \eqref{eq:FastSlowHJE} and the variational representation for the solutions. The convergence of viscosity solutions is obtained via a Lipschitz estimate uniformly in $\eps$ and the convergence of the variational representation follows from the $\Gamma$-convergence of the action functional.   Particularly, these two convergences are matched together on the slow manifold $\MS=\cE_\fast\cap \overline{\Omega}$. We finally obtain that the limiting variational representation solves the effective HJE on the slow manifold $\MS$.    Moreover, we use the reconstruction map $\sR$ to express the effective functionals and the variational representation by coarse-grained variables.  

\textbf{Main result: } Below, we state our   main result in a concrete way, with reference to the conditions, notations, and explicit functionals provided later (see $L_\eff$ in \eqref{eq:EffectiveLagrangian}, $H_\eff$ in \eqref{eq:EffectiveHamiltonian}, Assumption   \ref{assumption} and   $\MS$ in \eqref{eq:MS}). 

\begin{thm*}
      Let $\Omega\subset \R^I$ be a  bounded open domain such that $\overline{\Omega} \subset \R_+^I$. Let initial data $u^\eps_0$ be well-prepared satisfying Assumption \ref{assumption}, and assume    both conditions \eqref{FDB} and \eqref{eq:UFEC}. Then the viscosity solutions $u^\eps$ to  \eqref{eq:FastSlowHJE} converges   to $u^*$ uniformly on $K\times[0,T]$, for any compact subset $K\subset\MS^o$. Moreover, for any $(x,t)\in \MS\times(0,+\8)$, the variational representation of viscosity solution 
   $$u^{\eps}(x,t)=\inf\left\{ u_{0}^\eps(\x(0))+\int_0^t L_\eps(\mathrm{x}(s), \dot{\x}(s))\ud s  :{ \x\in\mathrm{AC}([0,t],\overline{\Omega})},\x(t)=x\right\}$$ 
   converges   to the limiting variational representation
   $$u^*(x,t)=\inf \big\{u_0^*(\mathrm{x}(0)) +  
   \int_0^t L_\eff(\mathrm{x}(s), \dot{\mathrm{x}}(s))\ud s: \x\in\mathrm{AC}([0,t],\MS),\x(t)=x \big\}.$$
   Last, $u^*$ solves the effective HJE in the viscosity sense
   \begin{equation}\label{HJE0}
    \begin{aligned}
\pt_t u(x,t) + H_\eff(x,d_xu(x,t))\leq 0, \quad (x,t)\in \MS^o\times(0,+\8),\\
\pt_t u(x,t) + H_\eff(x,d_x u(x,t))\geq 0, \quad (x,t)\in \MS\times(0,+\8),\\
u(x,0) = u_0(x), \quad x\in \MS=\overline{\MS}. 
\end{aligned}     
   \end{equation}
\end{thm*}

The main theorem will be proved via the following procedures. In Section \ref{sec2}, after reviewing the motivation and basic setup for RRE and HJEs, we state both conditions   \eqref{FDB} and  \eqref{eq:UFEC}, define necessary spaces separating fast-slow dynamics, and introduce the reconstruction map $\sR$. The effective dynamics on the slow manifold $\MS$ will then be  equivalently connected to the coarse-grained or projected dynamics.  In Section \ref{sec3}, we explicitly compute the fast-slow Hamiltonian and Lagrangian with some lower bound estimates. We also compute the effective and coarse-grained Hamiltonian and Lagrangian as a preparation for later sections. In Section \ref{s:GammaConvergence}, we prove the $\Gamma$-convergence of $\eps$-dependent action functional of    general absolutely continuous curves on any bounded subset of $\R_{\geq 0}^I$ under the two conditions \eqref{FDB} and \eqref{eq:UFEC}.    The $\Gamma$-convergence result consists of the compactness for curves with finite action (see Proposition \ref{prop:Compactness}), the lower bound estimates for the action functional for any curve that converges weakly in $\L^1([0,T];\R_{\geq 0}^I)$ (see Proposition \ref{prop:liminf}), and the upper bound estimate via the construction of a strongly convergent recovery sequence (see Proposition \ref{prop:ConstructionRecoverySequence}). The resulting effective action functional regulates the trajectory staying on  {   the set of fast equilibria $\cE_\fast$.}  In Section \ref{sec5}, we prove the convergence of viscosity solutions $u^\eps$ to the state-constraint \eqref{eq:FastSlowHJE} in Corollary \ref{cor:ExistenceLimitHJE}. Due to the degeneracy of $H_\eps$ in the momentum variable $p$ in some directions, it is necessary to establish  the coercivity and Lipschitz estimates in a non-degenerate subspace. {   Due to the lack of uniform coercivity near the boundary, we also consider state-constraint solutions $u^\eps$ on the domain $\Omega$ that has   a positive distance away from $\pt \R_+^I$}. Based on   well-prepared initial data with uniform $C^1$ bounds and a vanishing gradient in the direction of fast reactions, we first obtain the uniform Lipschitz estimates in time-space for $u^\eps$ separately for fast/slow directions (see Theorem \ref{thm:HJElimit}).   This ensures the uniform convergence to a limiting solution $u^*$.
In Section \ref{sec6}, we connect the convergence of solutions of \eqref{eq:FastSlowHJE} and the $\Gamma$-convergence of the solution representation to finally identify the limiting solution $u^*(x,t)$ for $x\in \MS$ and its variational representation via the effective action functional (see Proposition \ref{prop:identify}). This limiting characterization proves that $u^*$ solves the effective HJE \eqref{HJE0}. 

 In summary, we obtain the convergence of \eqref{eq:FastSlowHJE}, without the reversibility assumption for the chemical reactions, and thus allowing applications to non-equilibrium biochemical systems. The justification of this convergence not only helps the identification of slow manifold and the effective RRE dynamics, more rich information comes along including the limiting behavior of the fluctuation estimates and the fluctuations for the effective dynamics itself.    The developed $\Gamma$-convergence approach and Hamilton-Jacobi method for characterizing effective non-typical dynamics and fluctuation information does not require reversibility or linearity, so we believe it can be extended to other stochastic systems.

In previous works, $\Gamma$-convergence has been used to derive effective systems in the context of fast-slow chemical reaction systems. For gradient-flows, $\Gamma$-convergence for the functionals that define a gradient structure is derived in \cite{MielkeStephan} for linear reaction systems, in \cite{MielkePeletierStephan} for nonlinear reversible reaction systems and in \cite{Stephan21} for linear reaction-diffusion systems. There, the detailed-balance assumption is crucial to apply gradient flow approaches.  {   The $\Gamma$-convergence result from Theorem \ref{thm:GammaConvergence} is an extension to irreversible reactions under the same assumption \eqref{eq:UFEC} as in \cite{MielkePeletierStephan}    and an additional condition \eqref{FDB}.}
In terms of large deviations for multiscale chemical reactions, the formal WKB expansion approach has been extensively studied and applied in biochemical systems, cf. \cite{PhysRevLett, buice2010systematic, bressloff2013metastability}.   Rigorous large deviation results for chemical reaction processes are limited to specific two-scale system \cite{li2017large} or under linear reaction assumptions.   $\Gamma$-convergence for the large-deviation rate functional of both concentration and fluxes is shown in \cite{PeletierRenger} for general irreversible but linear fast-slow reactions.  
In terms of singular limit of general HJEs raising from optimal control problem with two time scales, there are thorough studies via similar PDE approach such as asymptotic expansions or $\Gamma$-convergence, cf. \cite{bardi1997optimal, alvarez2007multiscale}. However, due to the separation of fast-slow domain variables in HJEs with or without periodic assumptions, those fall under homogenization problems, so reviewing such literature goes beyond the scope of our paper.

The remaining paper is outlined as follows. 
In Section \ref{sec2}, we review the basic setup and motivations for studying the singular limit of \eqref{eq:FastSlowHJE}. In Section \ref{sec3}, we compute explicitly the fast-slow, effective and coarse-grained Hamiltonian/Lagrangian.  In Section \ref{s:GammaConvergence}, we prove the $\Gamma$-convergence of $\eps$-dependent action functional to the effective one. In Section \ref{sec5}, we obtain the convergence of the viscosity solution to \eqref{eq:FastSlowHJE}. In Section \ref{sec6}, we obtain the variational representation for the limiting viscosity solution and identify the effective HJE.

\section{Setup and basic properties for $\eps$-dependent fast-slow chemical reaction system}\label{sec2}
In this section, we review the basic setup for $\eps$-dependent fast-slow chemical reaction system with the associated \eqref{eq:FastSlowRRE} and \eqref{eq:FastSlowHJE}. First, we introduce definitions and properties of fast-slow RRE with conserved quantities. We also introduce the fast-detailed balance assumption and provide preliminary propositions on characterization of the effective RRE and the reconstruction map. Finally, we derive the HJE using WKB expansion for the chemical master equation.
\subsection{Chemical reaction systems}\label{ss:CRS}
Chemical reactions involving $I$  species $X_i$, $i\in \cal I =\{1,\cdots,I\}$ and $R$ reactions, indexed as $r\in \cal R=\{1,\cdots,R\}$, can be kinematically described as
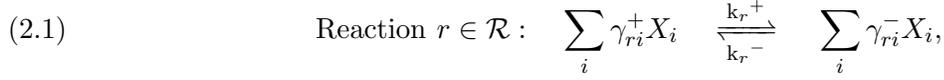
\begin{equation}\label{CRCR}
\text{reaction }r\in\cal R: \quad \sum_{i} \gamma_{ri}^+ X_{i} \quad \ce{<=>[k_r^+][k_r^-]} \quad   \sum_i \gamma_{ri}^- X_i,
\end{equation}
where the nonnegative integers $\gamma_{ri}^\pm \geq 0$ are stoichiometric coefficients and $k_r^\pm\geq 0$ are the reaction rates for the $r$-th forward and backward reactions. Throughout the paper we will assume the following type of weak-reversibility: $\forall r\in \cal{R}$ we have $k_r^+=0$ if and only if $k_r^-=0$. The column vector $\gamma_r:= \gamma_r^- - \gamma_r^+ := \bbs{\gamma_{ri}^- - \gamma_{ri}^+}_{i=1:I}\in \bZ^I$ is called the reaction vector for the $r$-th reaction, counting the net change in molecular numbers for species $X_i$.     For notational convenience, we introduce the so-called \textit{Wegscheider matrix} $\WM\in\R^{R\times I}$ defined by
\begin{equation}\label{eq:WM}
    \WM\in\R^{R\times I},\quad \WM_{r,i} := \gamma_{r,i}.
\end{equation}
  
Let $\mathbb{N}$ be the set of natural numbers including zero. In this paper, all vectors $X= \bbs{X_i}_{i=1:I} \in \mathbb{N}^I$ and $\bbs{\Psi_r}_{r\in \cal R},\, \bbs{k_r}_{r\in \cal R}\in \bR^R$ are column vectors.
Let $\mathbb{N}$ be   the state space for the counting process $X_i(t)$, representing the number of each species $i=1,\cdots,I$ in the biochemical reactions. With the reaction container size $1/h\gg 1$, the process $\cv_i(t)=hX_i(t)$     can be modeled as  the random time-changed Poisson representation for chemical reactions  (see \cite{kurtz1980representations, Kurtz15}):
\begin{align}\label{CR}
\cv(t) = \cv(0) + \sum_{r\in\cal R}  \gamma_{r} h \Bigg(  Y^+_r  \bbs{\frac{1}{h}\int_0^t {\Psi}^+_r(\cv(s))\ud s} 
- Y^-_r  \bbs{\frac{1}{h}\int_0^t {\Psi}^-_r(\cv(s))\ud s}\Bigg).
\end{align}
Here, for the $r$-th reaction channel, $Y^\pm_{r}(t)$ are i.i.d. unit-rate Poisson processes, and $\Psi_r^\pm(\cv)$ is the  intensity function. We will use the common form for the intensity function $\Psi_r^\pm(\cv)$ that is given by the macroscopic law of mass action (LMA):
\begin{equation}\label{lma}
\Psi^\pm_r(\vx) = k^\pm_r  \prod_{i=1}^{I} \bbs{x_i}^{\gamma^\pm_{ri}}.
\end{equation}

Macroscopically, the change of concentration can be described by the so-called \textit{reaction-rate equation} (RRE), given by
\[
\dot{x}(t)=R(x(t)):=\sum_{r\in{\cal R}}\left(\Psi_{r}^{+}(x)-\Psi_{r}^{-}(x)\right)\gamma_{r},\quad\Psi_{r}^{\pm}(x)=k_{r}^{\pm}\prod_{i}x_{i}^{\gamma_{i}^{\pm,r}},
\]
where the natural state space of concentrations is given by $\sC:=\bR^I_{\geq 0}:= \{x\in \bR^I;\, x_i \geq 0 \text{ for all } i\}$ as the largest living space for $x$. We also recall $\bR^I_+= \{x\in \bR^I;\, x_i > 0 \text{ for all } i\}$.

Although the (RRE) is not our primary objective of investigations, we collect here some properties, which will also be  important later. 
Important quantities that help to analyze the (RRE) are the so-called \textit{conserved quantities} $q\in\R^{I}$ such that $q\in\Gamma^{\perp}=\left\{ q\in\R^{I}:\forall\gamma\in\Gamma:q\cdot\gamma=0\right\} $, where  we define the stoichiometric subspace $\Gamma:=\mathrm{span}\left\{ \gamma_{r}:r\in{\cal R}\right\} $.
(Note that we do not assume that they are linearly independent.) 
In particular, for the RRE, $q\cdot x(t)$ is not changed along the evolution (which also clarifies their name).
Fixing a basis $q_{1},\dots,q_{m}$ of $\Gamma^{\perp}$, we define
the matrix $Q\in\R^{m\times I}$ by its adjoint $Q^{T}=(q_{1},\dots,q_{m})\in\R^{I\times m}$.
By construction, $Q^{T}$ is injective, $Q$ is surjective and $\mathrm{ker}(Q)=\Gamma$.


\subsection{Fast-slow RRE and conserved quantities}\label{ss:FSRRE}

  As explained in the introduction, we 
 assume the rate $k_r$ for all reactions appear to be only two scales. That is, the fast reaction has a $1/\eps$ order  while the slow reaction has  order $O(1)$.  To be more precise, we split
\[
{\cal R}={\cal R}_{\slow}\,  \cup\, {\cal R}_{\fast}
\]
such that 
\[
k_{r}^{\pm}=k_{r}^{\pm}(\eps)=\begin{cases}
\frac{1}{\eps}k_{r}^{\pm}, & r\in{\cal R_{\fast}}\\
k_{r}^{\pm}, & r\in{\cal R_{\slow}}.
\end{cases}
\]
Then the fast-slow RRE becomes \eqref{eq:FastSlowRRE}, i.e.,
\begin{align} 
\dot{x}(t) & =R_{\slow}(x(t))+\frac{1}{\eps}R_{\fast}(x(t))  :=\sum_{r\in{\cal R}_{\slow}}\left(\Psi_{r}^{+}(x)-\Psi_{r}^{-}(x)\right)\gamma_{r}+\frac{1}{\eps}\sum_{r\in{\cal R}_{\fast}}\left(\Psi_{r}^{+}(x)-\Psi_{r}^{-}(x)\right)\gamma_{r}.
\end{align}

 Similar to the general system, we
define the fast stoichiometric subspace $\Gamma_{\fast}=\mathrm{span}\left\{ \gamma_{r}:r\in{\cal R}_{\fast}\right\} $.
Then, we have $\Gamma^{\perp}\subset\Gamma_{\fast}^{\perp}$, and
$m_{\fast}=\mathrm{dim}\Gamma_{\fast}^{\perp}\geq m=\mathrm{dim}\Gamma^{\perp}$.
Extending the basis $q_{1},\dots,q_{m}$ of $\Gamma^{\perp}$ to a
basis $q_{1},\dots,q_{m_{\fast}}$of $\Gamma_{\fast}^{\perp}$, we
define the conservation matrix $Q_{\fast}:\R^{I}\to\R^{m_{\fast}}$
by $Q_{\fast}^{T}=(q_{1},\dots,q_{m_{\fast}})\in\R^{I\times m_{\fast}}$.
In particular, we have
\[
\mathrm{ker}Q_{\fast}=\Gamma_{\fast},\quad\mathrm{range}Q_{\fast}^{T}=\Gamma_{\fast}^{\perp}.
\]
Moreover, we set 
\begin{equation}\label{sQ}
\sQ:=\left\{ Q_{\fast}x\in\R^{m_{\fast}}:x\in\dom\right\}.    
\end{equation}
In the following we will call, elements $\sq\in\sQ$ \textit{fast conserved quantities} because they are the slow dynamic variables and do not change for the fast part of the evolution.    For future reference, we also define the \textit{fast} Wegscheider matrix for the fast conserved quantities
\begin{equation}\label{eq:WMfast}
    \WM_\fast\in\R^{R\times m_\fast},\quad \WM_\fast=\WM Q^T_\fast.
\end{equation}

Heuristically,
for small times the fast part $R_{\fast}(c(t))$ of \eqref{eq:FastSlowRRE} will dominate, while
for larger times ($t>\eps$) the slow reactions drive the evolution
and the fast parts are in equilibrium.
Thus we define the set of fast equilibria  that defines later the  slow manifold of the evolution:
\[
{\cal E}_{\fast}:=\left\{ x\in\dom:\quad\forall r\in{\cal R}_{\fast}:\Psi_{r}^{+}(x)=\Psi_{r}^{-}(x)\right\} .
\]
For consistency, we cite the following classical convergence result, which goes back to D. Bothe \cite{bothe2003reaction}.
\begin{thm}[\cite{bothe2003reaction}]\label{thm:EffectiveRRE}
    Let $x^\eps$ be the solution of the fast-slow   \eqref{eq:FastSlowRRE} with initial values $x^\eps(0)=x^\eps_0>0$. Assume that $x^\eps_0\to x^*_0>0$. Then we have $x^{\eps}\to x^{*}\in C([0,T],\R_{+}^{I})$, where $x^*$ solves the following effective RRE: we have $x^{*}(t)\in{\cal E}_{\fast}$ for $t>0$ and its evolution is described by the slow reactions only and a Lagrange parameter that
forces the evolution to stay on that set, i.e.
 \begin{equation}\label{effectiveRRE}
\dot{x}^{*}=R_{\slow}(x^{*}(t))+\lambda(t),\quad x^{*}(t)\in{\cal E}_{\fast},\quad\lambda(t)\in\Gamma_{\fast},\quad x^{*}(0)=x_{0}^{*}.
 \end{equation} 
\end{thm}
  This theorem characterizes the effective evolution. Moreover, we will see in Proposition \ref{prop:DifferentWaysEffectiveRRE} that the effective evolution can also be described in other ways by using projections, or coarse-grained variables.

\subsection{Assumptions on the fast reaction part and the reconstruction map}
Instead of analyzing the \eqref{eq:FastSlowRRE}, we are going to investigate the fluctuations around the typical path. For that we need two assumptions on the reaction system. The first assumption states that the fast reaction part $x'=R_\fast(x)$ is a detailed balanced reaction system. The second assumption is a condition on the fast equilibria $\cE_\fast$, which allows a resolution of   $\cE_\fast$ by the fast conserved quantities $\sq\in\sQ$.

\begin{asmp}[Fast-detailed balance assumption]\label{ass:FDB}
    Suppose there is a positive equilibrium state $x_{s}^{*}\in\R^I_+$, such that the fast reactions are in detailed balance, i.e.
\begin{equation}\label{FDB}
\exists x_{s}^{*}\in\R_{+}^{I}\,\,  \text{ such that }\,\, \Psi_{r}^{+}(x_{s}^{*})=\Psi_{r}^{-}(x_{s}^{*}),\,\, \forall r\in{\cal R}_{\fast}. \tag{{\sf FDB}}
\end{equation}
\end{asmp}

This assumption implies that the fast reaction part $x'=R_\fast(x)$ is a detailed-balanced reaction system with equilibrium $x_{s}^{*}$. Note that $x_{s}^{*}$ may not be an equilibrium of the whole system.

A direct calculation shows that the fast detailed balance condition \eqref{FDB}, can be rewritten in the following equivalent
ways, which will be useful later. It is also convenient to define the   subset of positive  fast equilibria   by 
\[
\mathcal{E}_{\fast,+}:=\mathcal{E}_\fast\cap\R_+^I.
\]
\begin{lem}\label{lem:FDB}
Suppose \eqref{FDB} holds. Then
\begin{enumerate}
\item $\forall r\in{\cal R}_{\fast}:\,\, \Psi_{r}^{+}(x_{s}^{*})=k_{r}^{+}\prod_{i}x_{s,i}^{*\gamma_{i}^{+,r}}=k_{r}^{-}\prod_{i}x_{s,i}^{*\gamma_{i}^{-,r}}=\Psi_{r}^{-}(x_{s}^{*})$ and $\log\left(k_{r}^{+}/k_{r}^{-}\right)=\gamma_{r}\cdot\log x_{s}^{*}$.
\item $\forall r\in{\cal R}_{\fast},   \forall x\in\R^I_+  :\gamma_{r}\cdot\log(x/x_{s}^{*})=\log(\Psi_{r}^{-}(x)/\Psi_{r}^{+}(x))$. 
\end{enumerate}
In particular, we have the following equivalent characterization of the set of fast equilibria:
\begin{equation}\label{eq:FDB}
 x\in \mathcal{E}_{\fast,+} \quad \Leftrightarrow \quad \forall r\in{\cal R}_{\fast}:\gamma_{r}\cdot\log(x/x_{s}^{*})=0 \quad \Leftrightarrow \quad \D\mathcal{H}(x|x_s^*)\in\Gamma^T_\fast ,   
\end{equation}
where $\mathcal{H}(x|x_s^*)$ is the relative entropy, i.e., $\mathcal{H}(x|y):=\sum_{i=1}^I x_i\log (x_i/y_i)-x_i+y_i$.

\end{lem}

   Moreover, we impose the following non-trivial structural assumption on the set of fast equilibria $\cE_\fast$. The same assumption has been stated in \cite{MielkePeletierStephan} for the $\Gamma$-convergence result for fast-slow detailed balance chemical reaction systems. Under Assumption \ref{ass:FDB}, it adapts to our situation as well.

\begin{asmp}[Unique fast equilibrium condition] For all $\sq\in\sQ = \{Q_\fast x |x\in\sC\}$ there is exactly one equilibrium of $x'=R_\fast(x)$ in the invariant subset $\{x\in\sC|Q_\fast x=\sq\}$, i.e.
\begin{equation}\label{eq:UFEC}
    \forall \sq\in\sQ:\quad \sharp (\{x\in\sC|Q_\fast x=\sq\} \cap \cE_\fast)=1 \tag{{\sf UFE}}.
\end{equation}
By $\sR:\sQ\to\sC$ we denote the mapping, s.t. $\{\sR(\sq)\}=\{x\in\sC|Q_\fast x=\sq\} \cap \cE_\fast$.    
\end{asmp}
The map $\sR$ will be called \textit{reconstruction map} because it resolves the fast equilibria $\cE_\fast$, where the slow evolution takes place, in terms of the fast conserved quantities of the evolution.

In addition (as in \cite{MielkePeletierStephan}), we impose the following positivity assumption on $\sR:$
\begin{equation}\label{eq:PositivityR}
    \exists \bar \sq\in\sQ:\forall\theta\in]0,1], \forall \sq\in\sQ, \forall i\in\mathcal{I}:\quad \sR(\sq+\theta\bar\sq)_i>0 \text{ and } \sR(\sq+\theta \bar \sq)_i\geq \sR(\sq)_i.
\end{equation}
We note that the positivity assumption is only needed in one technical step, namely for the construction of the recovery sequence for trajectories ``touching" the boundary $\partial\sC$. (Note,  considering the setting for the Hamilton-Jacobi equation in Section \ref{sec5}, where the underlying domain $\Omega$ has   a positive distance to the boundary, this assumption is not needed.) We also expect that with a more careful analysis in Proposition \ref{prop:ConstructionRecoverySequence} the positivity assumption can be discarded.

The \eqref{eq:UFEC} stressed the difference between the typical path and the fluctuations. The typical path starting from a positive concentration vector converges in time to the minimizer of the relative entropy $\mathcal{H}(\cdot |x_s^*)$ in the invariant set $\{x\in\sC|Q_\fast x=\sq\}$. Other possible equilibria are on the (repelling) boundary $\partial\sC$, which the \eqref{eq:UFEC} excludes. The \eqref{eq:UFEC} really means that $\mathrm{im}(\sR)=\cE_\fast$.  The necessity of the \eqref{eq:UFEC} for our result comes from the fact that general fluctuations do not satisfy the maximum principle. Note that the \eqref{eq:UFEC} fails for autocatalytic reactions. The loss of compactness in that situation is discussed in \cite[Rem. 3.10]{MielkePeletierStephan}. 

The reconstuction map $\sR$ has been already introduced and analyzed in \cite{MielkePeletierStephan}. Here, we only summarize the important properties and refer to \cite{MielkePeletierStephan} for more details and the proofs.

\begin{prop}[\cite{MielkePeletierStephan}]\label{prop:PropertiesReconstructionAndProjection}
Let   \eqref{FDB} and \eqref{eq:UFEC} be satisfied. Then:
\begin{enumerate}
\item The function $\sR:\sQ\to\dom$ is continuous, and $\sR:\mathrm{int}\sQ\to\mathrm{int}\sC$ is analytic.
\item We have $Q_{\fast}\sR(\sq)=\sq$ and $\mathrm{im}(\sR)=\mathcal{E}_\fast$.
\item For all $\sq\in\mathrm{int}\sQ$ we have $Q_{\fast}\D\sR(\sq)=I_{m_{\fast}}$,
and $Q_{\fast}^{T}\D\sR(\sq)^{T}$ is a projection on $\mathrm{im}(Q_{\fast}^{T})=\Gamma_{\fast}^{\perp}$. 
\item For all $x\in\mathcal{E}_{\fast,+}$ the tangent space of the manifold of equilibria is given by $\mathrm{T}_{x}\mathcal{E}_{\fast}=\mathbb{H}(x)^{-1}\Gamma_{\fast}^{\perp},$
where $\mathbb{H}(x):=\D^{2}{\cal H}(x|x_{s}^{*})=\mathrm{diag}(1/x_{1},\dots,1/x_{I})$.
\item Define the projection $\mathbb{P}(x)$ by $\mathrm{im}\mathbb{P}(x)=\Gamma_{\fast}$,
$\mathrm{ker}\mathbb{P}(x)=\mathbb{H}^{-1}(x)\Gamma_{\fast}^{\perp}$.
Then, we have the identity $(I-\mathbb{P}(\sR(\sq)))=\D\sR(\sq)Q_{\fast}$
and, clearly, the operator $\mathbb{P}(x)$ is uniformly bounded on compact subsets of $K\subset\R^I_+$.
\end{enumerate}
\end{prop}

Using the reconstruction map $\sR:\sQ\to\sC$, we can now equivalently describe the effective evolution in Theorem \ref{thm:EffectiveRRE}.
\begin{prop}\label{prop:DifferentWaysEffectiveRRE}
    Let $x(0)\in{\cal E}_{\fast,+}
$. Then the effective dynamics \eqref{effectiveRRE} from Theorem \ref{thm:EffectiveRRE} can be equivalently described in two additional ways:
\begin{enumerate}
    \item Projected dynamics: $\dot{x}=(I-\mathbb{P}(x))R_{\slow}(x)$.
    \item Coarse-grained dynamics: $\dot{\sq}(t)=Q_{\fast}R_{\slow}(\sR(\sq(t)))$.
\end{enumerate}
\end{prop}


\subsection{Fast-slow HJE for non-typical paths resulted from the large deviation principle}\label{ss:LDP}
In order to study the large fluctuations in the chemical reaction process \eqref{CR} that deviate from the typical path described by RRE, the WKB expansion in the chemical master equation (Kolmogorov forward equation) is a commonly used method, cf. \cite{Kubo73, Hu87, QianGe17, QianBook}.
The time marginal law of $\cv(t)$, denoted  as $p_{h}(x_i,t)$, for $x_i$ in a discrete domain, 
satisfies  forward   equation   \cite{Kurtz15, GL23} 
\begin{equation}\label{rp_eq}
\begin{aligned}
\frac{\ud}{\ud t} p_{h}(x_i, t) =& \frac{1}{h}\sum_{r\in \cal R}  \left( \Psi^+_r(x_i- \gamma_{r} h) p_{h}(x_i- \gamma_r h,t) - \Psi^-_r(x_i) p_{h}(x_i,t) \right) \\
& + \frac{1}{h}\sum_{r\in \cal R} \left( \Psi^-_r(x_i+\gamma_r h) p_{h}(x_i+\gamma_rh,t) - \Psi_r^+(x_i) p_{h}(x_i,t) \right).
\end{aligned}
\end{equation}
The exponential tilt for the probability density $p_h(x_i)=\e^{- \frac{u_h(x_i)}{h} }$ provides a   Hamiltonian viewpoint for studying the non-typical paths of the fast-slow dynamics, which happens with exponentially small probability compared with the typical path described RRE, cf. \cite{Hu87, QianGe17, GL22}. Using the exponential tilt, $u_i$ satisfies a discrete HJE (nonlinear ODE)
\begin{align*}
\partial_t u_h(x_i,t) + \sum_{r\in \cal R}\big[  &\Psi^+_r(x_i-\gamma_r h )    \e^ {\frac{u_h(x_i,t) - u_h(x_i-\gamma_r h,t)}{h} }
-    \Psi^-_r(x_i)\big]  \\   
&+ \sum_{r\in \cal R} \big[\Psi^-_r(x_i+ \gamma_r h)     \e^{ \frac{u_h(x_i,t)-u_h(x_i+\gamma_r h,t)}{h} }-      \Psi^+_r(x_i)\big]=0.
\end{align*}
Formally, as $h\to 0$, Taylor's expansion of $u_h$ with respect to $h$ leads to the following HJE 
\begin{equation}\label{HJEpsi}
  \pt_t u({x}, t) +\sum_{r\in \cal R}     \bbs{ \Psi^+_r(x)\bbs{\e^{\gamma_r  \cdot \nabla u({x},t)}   -  1} +  \Psi^-_r(x)\bbs{\e^{-\gamma_r  \cdot \nabla u(x,t)}   - 1 }}=0.
 \end{equation}
 The rigorous proof of the large deviation principle using the convergence of discrete nonlinear semigroup to the continuous one can be done through the WKB expansion in the backward equation for process $\cv$ and the convergence of the resulted monotone scheme for the first order continuous HJE, which we refer to \cite{GL23}.

 In this paper we start from the continuous HJE \eqref{HJEpsi} with   fast-slow multiscale intensity functions: $\Psi_r^\pm$ with $r\in \cal R_\slow$ and $\Psi_{r,\eps}^\pm=\frac{1}{\eps} \Psi_r^\pm$ with $r\in \cal R_\fast.$ Then the Hamiltonian in \eqref{HJEpsi} becomes $H_\eps$ in \eqref{eq:Heps}. To study the limiting behaviors and the variational representation, we consider chemical reactions restricted in a bounded domain  $\Omega$, which satisfies $\overline{\Omega}\subset \R_+^I$. Then the associated HJE on a bounded domain with state-constraint boundary condition becomes \eqref{eq:FastSlowHJE}.
The initial condition $u^\eps(x,0) = u_0^\eps(x)$ will be specified later. The definition of viscosity solution to \eqref{eq:FastSlowHJE} will be given in Definition \ref{def:vis}.



\section{Fast-slow, effective, and coarse-grained  Hamiltonian and Lagrangian}\label{sec3}
Before we study the limit of  \eqref{eq:FastSlowHJE}, we first recall the functionals, i.e. the Hamiltonian and  the Lagrangian from the introduction, see \eqref{eq:Heps}. In particular, we state explicitly the fast-slow Hamiltonian $H_\eps$ and Lagrangian $L_\eps$ for the fast-slow chemical reaction system, followed by the effective Hamiltonian $H_\eff$ and Lagrangian $L_\eff$, which will be a part of the $\Gamma$-convergence result in Section \ref{s:GammaConvergence}. Finally, we define coarse-grained functionals $\sH$ and $\sL$, which will be the expression of  the effective functionals in coarse-grained variables. Note that throughout the section the functionals are defined on subsets of the Euclidean   space and the canonical inner product is denoted by $x\cdot p = \sum_i x_ip_i$.

\subsection{Fast-slow action Hamiltonian and Lagrangian}
 Throughout this section, $\eps>0$ is fixed.
Recall from \eqref{eq:Heps} the fast-slow Hamiltonian $H_\eps:\mathbb{R}^{I}\times\mathbb{R}^{I}\to\R$. By Legendre transform (see Lemma \ref{lem:DualityEpsFunctionals} below), the fast-slow Lagrangian is given by
\begin{align}\label{eq:Leps}
L_{\eps}(x,v) & =\inf_{v=\WM^{T}J}\sum_{r\in{\cal R}}S(J_{r}|\Psi_{r,\eps}^{+}(x),\Psi_{r,\eps}^{-}(x)),
\end{align}  
where we have used the Wegscheider matrix $\WM\in\R^{R\times I}$ from \eqref{eq:WM}. Note that we have 
$\left(\WM^{T}J\right)_{i}=\sum_{r}\gamma_{r,i}J_{r}$, and $\left(\WM p\right)_{r}=\gamma_{r}\cdot p$.  The equation 
$$v=\WM^T J$$ connects the change of rate $v$ with fluxes $J$ by a gradient-type operator $\WM$, and, in the following, will be called \textit{continuity equation}.  
Above, we define  for $\alpha,\beta>0$ the   function
\[
S(J|\alpha,\beta):=\sqrt{\alpha\beta}\cdot{\cal C}\left(\frac{J}{\sqrt{\alpha\beta}}\right)-J\frac{1}{2}\log(\alpha/\beta)+\left(\sqrt{\alpha}-\sqrt{\beta}\right)^{2},
\]
as the Legendre transform of the function $(\alpha,\beta,p)\mapsto S_{\alpha,\beta}^{*}(p) := \alpha\left(\e^{p}-1\right)+\beta\left(\e^{-p}-1\right)$
with respect to the variable $p\in\R$. Here ${\cal C}$ is the Legendre transform of the cosh-function ${\cal C}^{*}(p)=2(\cosh(p)-1)$. In Proposition \ref{prop:PropertiesS},  properties of the functions
$S$ and $S^{*}$ will be  summarized.
Before that, we first show that indeed the Hamiltonian $H_{\eps}$ and
the Lagrangian $L_{\eps}$ are dual to each other.

\begin{lem}\label{lem:DualityEpsFunctionals}
For all $x\in\sC$, $v\in\R^I$, we have $H_{\eps}(x,\cdot)^{*}(v)=\sup_{p}\left(p\cdot v-H_{\eps}(x,p)\right)=L_{\eps}(x,v)$.
\end{lem}
Note, that in particular, we have $L_\eps(x,v)=+\infty$ if $v\notin\Gamma$.   Recalling the definition of the matrix $Q$ of conserved quantities from Section \ref{ss:CRS}, this means that for a trajectory $t\mapsto x(t)$ with bounded $L_\eps(x(t),\dot x(t))$ we necessarily have that $Q\dot{x}=0$. 
\begin{proof}
For simplicity, we neglect the explicit $\eps$-dependence in the reaction
intensities $\Psi_{r,\eps}^{\pm}$. We compute the Legendre transform of $L_\eps$ from \eqref{eq:Leps} and obtain
\begin{align*} 
 & L_{\eps}^{*}(x,p)=\\
 & =\sup_{v}\left\{ v\cdot p-\inf_{v=\WM^{T}J}\sum_{r\in{\cal R}}S(J_{r}|\Psi_{r}^{+}(x),\Psi_{r}^{-}(x))\right\} =\sup_{v,J:v=\WM^{T}J}\left\{ v\cdot p-\sum_{r\in{\cal R}}S(J_{r}|\Psi_{r}^{+}(x),\Psi_{r}^{-}(x))\right\} \\
 & =\sup_{J}\left\{ \WM^{T}J\cdot p-\sum_{r\in{\cal R}}S(J_{r}|\Psi_{r}^{+}(x),\Psi_{r}^{-}(x))\right\} =\sup_{J}\left\{ J\cdot\WM p-\sum_{r\in{\cal R}}S(J_{r}|\Psi_{r}^{+}(x),\Psi_{r}^{-}(x))\right\} \\
 & =\sup_{J}\left\{ \sum_{r}J_{r}\left(\WM p\right)_{r}-\sum_{r\in{\cal R}}S(J_{r}|\Psi_{r}^{+}(x),\Psi_{r}^{-}(x))\right\} =\sum_{r\in{\cal R}}S^{*}_{\Psi_{r}^{+}(x),\Psi_{r}^{-}(x)}(\gamma_r\cdot p )=H_{\eps}(x,p). 
\end{align*}
\end{proof}

In the later sections we need some properties of the function $S$, which will be summarized and proved in the following proposition.
\begin{prop}
\label{prop:PropertiesS}Let $[0,\infty[\times[0,\infty[\times\R\ni(\alpha,\beta,p)\mapsto S_{\alpha,\beta}^{*}(p) := \alpha\left(\e^{p}-1\right)+\beta\left(\e^{-p}-1\right),$
and define its Legendre transform by $S(J|\alpha,\beta)=\sup_{p\in\R}\left(p\cdot J-S_{\alpha,\beta}^{*}(p)\right)$.
The function $S$ has the following properties:
\begin{enumerate}
\item For all $\alpha,\beta\geq0$, $J\in\R$, we have $S(J|\alpha,\beta)\geq0$
and $S(J|\alpha,\beta)=S(-J|\beta,\alpha)$.
\item For $\alpha,\beta>0$, we have the  the following equivalent characterizations
\begin{align*}
S(J|\alpha,\beta) & =\sqrt{\alpha\beta}\cdot{\cal C}\left(\frac{J}{\sqrt{\alpha\beta}}\right)-J\frac{1}{2}\log(\alpha/\beta)+\left(\sqrt{\alpha}-\sqrt{\beta}\right)^{2}\\
 & =\inf_{J=u-w}{\cal H}(u|\alpha)+{\cal H}(w|\beta),
\end{align*}
where ${\cal C}$ is the Legendre transform of the cosh-function ${\cal C}^{*}(p)=2(\cosh(p)-1)$ and $\cal H$ is the relative entropy.
In particular, we have
\[
S(J|\alpha,\beta)\geq-J\frac{1}{2}\log(\alpha/\beta)+\left(\sqrt{\alpha}-\sqrt{\beta}\right)^{2}.
\]
\item For $M>0$ satisfying $\alpha,\beta\leq M$, we have that $S(J|\alpha,\beta)\geq M{\cal C}\left(J/M\right)-2M$.
\end{enumerate}
\end{prop}

\begin{proof}
The first two claims follow directly from the definition by the Legendre
transform. For the third
claim, we observe that

\[
S_{\alpha,\beta}^{*}(p)=\alpha(\e^{p}-1)+\beta(\e^{-p}-1)\leq M\left(\e^{p}+\e^{-p}\right)\leq2M\cosh(p)=M({\cal C}^{*}(p)+2).
\]
Hence, 
\begin{align*}
S(J|\alpha,\beta) & =\sup_{p}\left\{ p\cdot J-S_{\alpha,\beta}^{*}(p)\right\} \geq\sup_{p}\left\{ p\cdot J-M({\cal C}^{*}(p)+2)\right\} \geq M\sup_{p}\left\{ p\cdot\frac{J}{M}-{\cal C}^{*}(p)\right\} -2M\\
 & =M{\cal C}\left(J/M\right)-2M.
\end{align*}
\end{proof}

\subsection{Effective Lagrangian and Hamiltonian}

As we will see from the $\Gamma$-convergence result (Theorem \ref{thm:GammaConvergence}), the effective Lagrangian $L_\eff$ is restricted to the slow reactions and takes into account the set of fast equilibria $\cal{E}_\fast$. It is defined as   
\begin{align}\label{eq:EffectiveLagrangian} 
L_{\mathrm{eff}}(x,v):=
\inf_{J}\left\{ \sum_{r\in{\cal R}_{\slow}}S(J_{r}|\Psi_{r}^{+}(x),\Psi_{r}^{-}(x)):\quad Q_{\fast}v = \WM ^T_\fast J\right\} 
\end{align}
for $x\in\cal{E}_\fast$,   where we have used the fast Wegscheider matrix $\WM_\fast=\WM Q^T_\fast$ from \eqref{eq:WMfast};    and $L_{\mathrm{eff}}(x,v):=+\infty$ if $x\notin\cal{E}_\fast$. Later we will restrict the domain of definition to $x\in\mathcal{E}_\fast$, see Section \ref{s:GammaConvergence}. Note that $L_\eff$ contains the slow part of the $\eps$-dependent Lagrangian without its fast-part. However the continuity equation is now contracted through the fast-conservation matrix $Q_\fast$.

The explicit expression of effective Hamiltonian $H_\eff$ is computed in the next lemma.
\begin{lem}\label{lem:HeffRepresentation}
For $x\in\cal{E}_\fast$, we have $L_{\mathrm{eff}}(x,\cdot)^{*}(p)=:H_{\mathrm{eff}}(x,p)$,
\begin{equation}\label{eq:EffectiveHamiltonian}
H_{\mathrm{eff}}(x,p)=\sum_{r\in{\cal R}_{\slow}}S_{\Psi_{r}^{+}(x),\Psi_{r}^{-}(x)}^{*}(\gamma _r\cdot p)+\chi_{\Gamma_{\fast}^{\perp}}(p),
\end{equation}
  where the characteristic function of convex analysis $p\mapsto\chi_A(p)$ is defined as 0, if $p\in A$ and $+\infty$ otherwise.  
In particular, we have that $H_{\eff}(x,p)=\infty$ if $p\notin\Gamma_{\fast}^{\perp}=\mathrm{range}Q_{\fast}^{T}$. Moreover, for all $x\in\cal{E}_\fast$ and $p\in\Gamma_{\fast}^{\perp}$ the Legendre transform $H_\eff(x,p)$ is attained by taking the supremum over $v\in T_x\cal{E}_\fast=\mathrm{Ker}(\mathbb{P}(x))$, i.e.
$$H_\eff(x,p)=\sup_{v\in T_x\cal{E}_\fast}\{p\cdot v - L_\eff(x,v)\},$$
 where the projection is defined by $\mathrm{Ker}(\mathbb{P}(x))=T_x\cal{E}_\fast$, $\mathrm{range}(\mathbb{P}(x))=\Gamma_\fast$ (see Proposition \ref{prop:PropertiesReconstructionAndProjection}). 
\end{lem}

\begin{proof}
Let us fix $x\in\cal{E}_\fast$. We compute the Legendre transform of $L_\eff$ and show that it is given by $H_\eff$. For this we observe first that 
\begin{align*}
L^{*}_\eff(x,p) & =\sup_{v}\left\{ v\cdot p-L_{\eff}(x,v)\right\} =\sup_{v,J:Q_{\fast}v=Q_{\fast}\WM^TJ }\left\{ v\cdot p-\sum_{r\in{\cal R}_{\slow}}S(J_{r}|\Psi_{r}^{+}(x),\Psi_{r}^{-}(x))\right\} \\
 & \geq\sup_{v:Q_{\fast}v=0}v\cdot p 
 =\chi_{\Gamma_{\fast}^{\perp}}(p).
\end{align*}
For $p=Q_{\fast}^{T}\sp$, a direct calculation shows that 
\begin{align*}
\left(L_{\mathrm{eff}}(x,\cdot)\right)^{*}(p=Q_{\fast}^{T}\sp) & =\sup_{v}\left\{v\cdot Q_{\fast}^{T}\sp-\inf_{J:Q_{\fast}v=Q_{\fast}\WM^TJ}\left\{ \sum_{r\in{\cal R}_{\slow}}S(J_{r}|\Psi_{r}^{+}(x),\Psi_{r}^{-}(x))\right\} \right\}\\
 & =\sup_{v:Q_{\fast}v=Q_{\fast}\WM^TJ}\left\{Q_{\fast}v\cdot\sp-\sum_{r\in{\cal R}_{\slow}}S(J_{r}|\Psi_{r}^{+}(x),\Psi_{r}^{-}(x))\right\}.
 \end{align*}
 The right-hand side depends on $v$ only via  $Q_\fast v$. However, using the projection $\mathbb{P}(x)$, we decompose $v=(I-\mathbb{P}(x))v+\mathbb{P}(x)v$, where $(I-\mathbb{P}(x))v\in T_x\cal{E}_\fast$ and observe $Q_\fast v = Q_\fast(I-\mathbb{P}(x))v$. Hence, it suffices to compute the supremum only for $v\in T_x\cal{E}_\fast$. Computing further we get 
 \begin{align*}
 \left(L_{\mathrm{eff}}(x,\cdot)\right)^{*}(p=Q_{\fast}^{T}\sp) & =\sup_{J}\left\{ Q_{\fast}\WM^TJ\cdot\sp-\sum_{r\in{\cal R}_{\slow}}S(J_{r}|\Psi_{r}{{}^+}(x),\Psi_{r}{{}^-}(x))\right\} \\
 & =\sup_{J}\left\{ \sum_{r\in{\cal R}_{\slow}}J_{r}\left(\WM p\right)_{r}+\sum_{r\in{\cal R}_{\fast}}J_{r}\left(\WM p\right)_{r}-\sum_{r\in{\cal R}_{\slow}}S(J_{r}|\Psi_{r}^{+}(x),\Psi_{r}^{-}(x))\right\} \\
 & =\sum_{r\in{\cal R}_{\slow}}S_{\Psi_{r}^{+}(x),\Psi_{r}^{-}(x)}^{*}(\gamma_r\cdot  p)=H_{\mathrm{eff}}(x,p).
\end{align*}
This proves the claim.
\end{proof}

\subsection{Coarse-grained
Hamiltonian and Lagrangian}\label{ss:CGGL}

As we have seen the effective Lagrangian and Hamiltonian both contain  constraints. These constraint can be resolved by using coarse-grained
variables. For this, recall the reconstruction map $\sR$ from \eqref{eq:UFEC}. 

The effective Lagrangian $L_\eff$ is finite if evaluated at $x=\sR(\sq)\in\mathcal{E}_\fast$. Moreover for the constraint in the Hamiltonian, we observe that $p\in\Gamma_{\fast}^{\perp}$
is equivalent to $p\in\mathrm{range}(Q_{\fast}^{T})$, i.e.
there is $\sp\in\R^{m_{\fast}}$ such that $p=Q_{\fast}^{T}\sp$.
Hence, we have
\begin{align}\label{eq:CGHamiltonian}
H_{\eff}(x,p=Q_{\fast}^{T}\sp) & =H_{\eff}(\sR(\sq),Q_{\fast}^{T}\sp)=\sum_{r\in{\cal R}_{\slow}}S_{\Psi_{r}^{+}(\sR(\sq)),\Psi_{r}^{-}(\sR(\sq))}^{*}(\gamma_{r}\cdot Q_{\fast}^{T}\sp)=\nonumber\\
 & =\sum_{r\in{\cal R}_{\slow}}S_{\Psi_{r}^{+}(\sR(\sq)),\Psi_{r}^{-}(\sR(\sq))}^{*}(Q_{\fast}\gamma_{r}\cdot\sp)=:\sH(\sq,\sp),
\end{align}
where we call $\sH=\sH(\sq,\sp)$ as the coarse-grained Hamiltonian on $\R^{m_{\fast}}\times\R^{m_{\fast}}$.
This Hamiltonian corresponds to a RRE with reaction intesities $\Psi_{r}^{\pm}\circ\sR$
and stochiometric vectors $Q_{\fast}\gamma_{r}\in\R^{m_{\fast}}$, meaning that it contains linear stochiometric vectors, however mass-action law might be violated (as already observed in \cite{MielkePeletierStephan}). The induced RRE of that Hamiltonian is then the coarse-grained
evolution equation
\[
\dot{\sq}=Q_{\fast}R_{\slow}(\sR(\sq(t))),
\]
which we have seen from Proposition \ref{thm:EffectiveRRE}.

For $x=\sR(\sq)\in\cal{E}_\fast$, the coarse-grained Lagrangian is then defined
by the following contraction:
\begin{align}\label{s:CGLagrangian}
\sL(\sq,\sv) :=& \inf_{v:\sv=Q_\fast v}L_\eff(\sR(\sq),v)\\
=&   \inf_{\sv=\WM^T_\fast J}\sum_{r\in{\cal R}_{\slow}}S(J_{r}|\Psi_{r}^{+}(\sR(\sq)),\Psi_{r}^{-}(\sR(\sq))). 
\end{align}
Indeed, computing the Legendre transform of the right-hand side, we get

\begin{align*}
    \sup_{\sv}&\{\sv\cdot \sp - \inf_{v:Q_\fast v = \sv} L_\eff(x,v)\} = \sup_{\sv,v:Q_\fast v=\sv}\{ \sv\cdot\sp - L_\eff(x,v)\}\\
    &= \sup_{v} \{Q_\fast v\cdot\sp - L_\eff(x,v)\} = \sup_{v} \{v\cdot Q_\fast ^T\sp - L_\eff(x,v)\} = H_\eff(\sq,Q_\fast^T\sp)=\sH(\sq,\sp),
\end{align*}
which shows that the convex and continuous functions $\sH$  and $\sL$ are both Legendre duals.

\section{$\Gamma$-convergence result of fast-slow action functional ${\cal L}_{\eps}$}\label{s:GammaConvergence}
In this section we state and prove the $\Gamma$-convergence result of fast-slow action functional ${\cal L}_{\eps}$, which will later be also applied to identify the limit of the HJE.
   To do so, we consider the full domain  $\sC$ (which is convex and closed) and we fix a time-interval $[0,T]$,
$T>0$. Moreover, we are interested in trajectories $\x:[0,T]\mapsto \x(t)$ with final point $\x(T)=x\in\sC$, fixed
throughout the section. To fix the functional analytic setting, we introduce the function space
\[
\mathrm{AC}_x([0,T],\sC):=\{\x\in\mathrm{AC}([0,T],\sC):\x(T)=x\}.
\]
The $\eps$-dependent Lagrangian $L_\eps$ defines the fast-slow action functional by
\begin{equation}
  {\cal L}_{\eps}(\x)= \begin{cases}
   \int_{0}^{T}L_{\eps}(\x(t),\dot{\x}(t))\,\d t, \quad &\x \in \mathrm{AC}_x([0,T],\sC), \\
   +\8, &\text{ otherwise.}
  \end{cases}   
\end{equation}
Later we will rely on the flux representation of the $\L^1$-velocity $\dot{\x}(t)=\WM^T J= \sum_{r}J_r(t)\gamma_{r}$    represented by general measurable fluxes $J:[0,T]\to\R^R$.    We first define our notion of $\Gamma$-convergence. Note that we restrict the curves to be bounded in accordance to Section \ref{sec5}, where we consider a bounded domain $\Omega\subset\R^I$.   
\begin{defn}\label{Def:Mosco}
We say $\mathcal{L}_\eps$ converges to $\mathcal{L}_\eff$ in the sense of    (bounded $\L^1$-)Mosco-convergence  , denoted as ${\cal L}_{\eps}\xrightarrow{M}{\cal L}_{\mathrm{eff}}$, if
(i)  a liminf-estimate holds for all weakly-converging sequences: 
\[
 \forall \x^\eps\rightharpoonup\x \text{ in } \L^1([0,T],\sC), \quad     \sup_{\eps>0}\|\x^\eps\|_{\L^\infty([0,T])}<\infty:   \quad \liminf_{\eps\to0}\mathcal{L}_\eps(\x^\eps )\geq \mathcal{L}_\eff(\x);
\]
and (ii) there exists a recovery sequence converging in $\L^1([0,T],\sC)$:
\[
\forall \x\in   \L^1([0,T],\sC), \quad \exists \x^\eps\to\x \text{ in } \L^1([0,T],\sC),\,     \sup_{\eps>0}\|\x^\eps\|_{\L^\infty([0,T])} <\infty:  \quad \lim_{\eps\to 0}\mathcal{L}_\eps(\x^\eps) = \mathcal{L}_\eff(\x).
\]
\end{defn}
{   We remark that the convergence of the recovery sequence can be improved to $C([0,T],\sC)$ (see Proposition \ref{prop:ConstructionRecoverySequence}.)}
  
\begin{thm}[$\Gamma$-convergence of action functional]\label{thm:GammaConvergence}
     Assume that both conditions \eqref{FDB} and \eqref{eq:UFEC} hold. Then, we have the Mosco-convergence ${\cal L}_{\eps}\xrightarrow{M}{\cal L}_{\mathrm{eff}}$, as defined in Definition \ref{Def:Mosco},
where 
 \begin{equation}\label{eq:EffectiveAF}
{\cal L}_{\eff}(\x):=\begin{cases}
\int_{0}^{T}L_{\mathrm{eff}}(\x(t),\dot{\x}(t))\d t, & \x\in C([0,T],\sC),\ \ Q_{\fast}\x\in\W^{1,1}([0,T],\R^{m_{\fast}}),\   \x(t)\in\cE_\fast,\\
+\infty & \text{otherwise,}
\end{cases}
 \end{equation}
and the effective Lagrangian $L_\eff$ is defined by \eqref{eq:EffectiveLagrangian}. In particular, continuous curves $\x$ with finite action $\cal{L}_\eff(\x)<\infty$ satisfy $\x(t)\in\cE_\fast$ for all $t\in[0,T]$.
\end{thm}


   Note that although the topology is given by $\L^1([0,T],\sC)$ the functionals $\mathcal{L}_\eps$ and $\mathcal{L}_\eff$ are only finite on continuous functions such that the point evaluation  $\x(T)=x$ makes sense.   
As usual the $\Gamma$-convergence is done in the following three steps: first showing
the necessary compactness, then the liminf-estimate, finally constructing the
recovery sequence.

\begin{prop}[Compactness]
\label{prop:Compactness}Let a sequence $\left(\x^{\eps}\right)_{\eps>0}$,
$\x^{\eps}\in\mathrm{AC}_x([0,T],\sC)$ satisfying ${\cal L}_{\eps}(\x^{\eps})\leq C$    and $|\x^\eps(t)|\leq C$   for some constant $C>0$
be given. 
Moreover, assume the \eqref{FDB} and the \eqref{eq:UFEC} conditions.
Then there exists a function $\x^{*}\in\L^{1}([0,T],\sC)$ with the following properties:
\begin{enumerate}
    \item $\left(Q_{\fast}\x^*\right)\in\W^{1,1}([0,T],\R^{m_{\fast}})$ and $\x^{*}(t)\in\cE_\fast$ for a.e. $t\in[0,T]$;
    \item $\x^{\eps}\to \x^{*}$ in any $\L^{p}([0,T],\R^{I})$, $p\in[1,\infty[$;
    \item $\x^{*}(t)=\sR(Q_{\fast}\x^{*}(t))$
for a.e. $t\in[0,T]$, which implies that $\x^{*}\in \W^{1,1}([0,T],\R^{m_{\fast}})$, i.e. $\x$ has an absolutely continuous representative.
\end{enumerate}
\end{prop}
   
Before we prove the proposition, we first need the following elementary estimate.
\begin{lem}\label{lem:EstimateCompactness}
    Let $x\in\sC$ be arbitrary, and let $\Psi^\pm(x)=k^\pm x^{\gamma^\pm}$ with $\gamma^\pm\in\N^I$ and $\gamma=\gamma^--\gamma^+$. Then for all $J\in\R$ it holds the pointwise estimate
    \begin{align}
        J \gamma\cdot \log (x/x_s^*) \leq  S(J|\Psi^+,\Psi^-) + k^+\bigg\{x^{\gamma^-}(x_s^*)^{-\gamma}-x^{\gamma^+} \bigg\} + k^- \bigg\{x^{\gamma^+}(x_s^*)^{\gamma}-x^{\gamma^-}\bigg\},
    \end{align}
    where the right-hand side is properly defined up to the boundary $x\in\partial \sC$.
\end{lem}
\begin{proof}
    By evaluating the Legendre transform of $S$ at $p=\log\bigg( (x/x_s^*)^{\gamma^--\gamma^+} \bigg)$, we get that
    \[
    S(J|\Psi^+,\Psi^-)\geq  J \log\bigg( (x/x_s^*)^{\gamma^--\gamma^+} \bigg) - \Psi^+((x/x_s^*)^{\gamma^--\gamma^+}-1) - \Psi^-((x/x_s^*)^{\gamma^+-\gamma^-} -1).   
    \]
    Using that $\Psi^+(x)=k^+x^{\gamma^+}$, $\Psi^-(x)=k^-x^{\gamma^-}$ and $\gamma=\gamma^- -\gamma^+$, we hence get an inequality of the form

     \begin{align*}
    J \gamma\cdot \log (x/x_s^*)&\leq S(J|\Psi^+,\Psi^-) + k^+x^{\gamma^+} \bigg\{x^{\gamma^--\gamma^+}(x_s^*)^{-\gamma}-1\bigg\} + k^-x^{\gamma^-} \bigg\{x^{\gamma^+-\gamma^-}(x_s^*)^{\gamma}-1\bigg\} \\
    &\leq  S(J|\Psi^+,\Psi^-) + k^+\bigg\{x^{\gamma^-}(x_s^*)^{-\gamma}-x^{\gamma^+} \bigg\} + k^- \bigg\{x^{\gamma^+}(x_s^*)^{\gamma}-x^{\gamma^-}\bigg\}.
    \end{align*}
    
\end{proof}

\begin{proof}[Proof of Proposition \ref{prop:Compactness}]
The proof is performed in several steps. To derive the time regularity,
we use the explicit formula $L_{\eps}(x,v)=\inf_{v=\WM^{T}J}\sum_{r\in{\cal R}}S(J_{r}|\Psi_{r}^{+}(x),\Psi_{r}^{-}(x))$. Hence, we may assume the existence of fluxes $J_{r}^{\eps}:[0,T]\to\R$ that
are measurable, such that $\dot{\x}^{\eps}(t)=\WM^T J^{\eps}(t)$
holds weakly in time and we have     by the non-negativity of $S$ that
\begin{equation}\label{eq:AprioriCompactness}
\forall r\in\mathcal{R}:\quad  \int_{0}^{T}S\big(J_{r}^{\eps}(t)|\Psi^{+}_{r,\eps }(\x^{\eps}(t)),\Psi^{-}_{r,\eps }(\x^{\eps}(t))\big)\d t\leq C.    
\end{equation}
  
{\bf Step 1:} Since by assumptions, the sequence $\x^\eps$ is uniformly bounded, we    have a limit  $\x^{*}\in\L^{\infty}([0,T],\sC)$ such that there exists a subsequence (not relabeled) $\x^{\eps}\overset{*}{\rightharpoonup} \x^{*}$ in $\L^\infty([0,T],\sC)$   as $\eps\to0$.    Note that a priori the limit $\x^*$ has no regularity.    \\
{\bf Step 2:} Compactness of slow fluxes:\\
Let $r\in{\cal R}_{\slow}$, and let $M:=\max\left\{ \Psi_{r}^{\pm}(x):r\in{\cal R}_{\slow},x\leq C\right\} <\infty$.
By Proposition \ref{prop:PropertiesS}, we have $C\geq\int_{0}^{T}\left(M{\cal C}\left(\frac{J_{r}^{\eps}}{M}\right)-2M\right)\d t$
which implies that $\left\{ \frac{J_{r}^{\eps}}{M}\right\} _{\eps>0}$
is uniformly bounded in the Orlicz space $\L^{{\cal C}}([0,T])$. Because
it is a Banach space we also get that 
\begin{align}\label{eq:FluxesInLC}
    J_{r}^{\eps} \text{ is uniformly bounded in }\L^{{\cal C}}([0,T]),
\end{align}
and hence we get weak compactness 
\begin{align}\label{eq:CompactnessFluxes}
    J_{r}^{\eps}\rightharpoonup J_r^*\quad \text{ in }\L^{1}([0,T]) \text{ for all } r\in\cal{R}_\slow. 
\end{align}
{\bf Step 3:} By the continuity equation, we get 
\[
Q_{\fast}\dot{\x}^{\eps}=Q_{\fast}\sum_{r\in{\cal R}}J_{r}^{\eps}\gamma_{r}=Q_{\fast}\sum_{r\in{\cal R}_{\slow}}J_{r}^{\eps}\gamma_{r},
\]
because $\Gamma_{\fast}=\mathrm{ker}Q_{\fast}$. Hence, we conclude
that $\left(Q_{\fast}\dot{\x}^{\eps}\right)\in\L^{{\cal C}}([0,T],\R^{m_{\fast}})$,
which implies by the Arzela-Ascoli theorem uniform convergence on $[0,T]$ of $Q_{\fast}\x^{\eps}=:\sq^{\eps}$
 to a continuous curve $\sq^{*}:=Q_{\fast}\x^{*}$. Here we have used that the terminal point $\x(T)$ is fixed. In particular, we also have strong convergence $Q_{\fast}\x^{\eps}\to \sq^*$ in $\L^1([0,T],\R^{m_\fast})$.\\
{\bf Step 4:} Next, we show that $\x^{*}(t)\in\mathcal{E}_\fast$ for almost all $t\in[0,T]$.      
For this we exploit the specific form of $S$ and rely on \eqref{FDB}. First, we observe that by defining $\tilde\x^\eps(t)=(\x^\eps\ast\psi^\eps)(t)$, where $\psi^\eps$ is a mollifier in time (i.e. on $[0,T]\subset \R$) approximating the identity as $\eps\to0$, the same a priori estimate \eqref{eq:AprioriCompactness} hold true for $\tilde \x^\eps$, and thus we may assume that $\x^\eps$ is strictly positive for any $\eps>0$ (but of course not uniformly in $\eps$). Using \eqref{eq:lb} and Proposition
\ref{prop:PropertiesS}, we have for all $r\in{\cal R}_{\fast}$ a
bound on 
\begin{align*}
 & \int_{0}^{T}S(J_{r}^{\eps}|\Psi^{+}_{r,\eps}(\x^{\eps}),\Psi^{-}_{r,\eps}(\x^{\eps}))\d t\\
 & \geq\int_{0}^{T}-J_{r}^{\eps}\frac{1}{2}\log(\Psi^{+}_{r,\eps}(\x^{\eps})/\Psi^{-}_{r,\eps}(\x^{\eps}))+\frac 1 \eps \left(\sqrt{\Psi^{+}_{r,\eps}(\x^{\eps})}-\sqrt{\Psi^{-}_{r,\eps}(\x^{\eps})}\right)^{2}\d t\\
 & =\int_{0}^{T}\frac{J_{r}^{\eps}}{2}\gamma_{r}\cdot\log(\x^{\eps}/x_{s}^{*})+\frac{1}{\eps}\left(\sqrt{\Psi^{+}_{r,\eps}(\x^{\eps})}-\sqrt{\Psi^{-}_{r,\eps}(\x^{\eps})}\right)^{2}\d t,
\end{align*}
where we have used the \eqref{FDB} and Lemma \ref{lem:FDB} in the last step.
   
Summing
over all $r\in{\cal R}_{\fast}$, we get the uniform bound
\begin{align*}
&\frac{1}{2}\int_{0}^{T}\sum_{r\in{\cal R}_{\fast}}J_{r}^{\eps}\gamma_{r}\cdot\log(\x^{\eps}/x_{s}^{*})\d t+\frac{1}{\eps}\sum_{r\in{\cal R}_{\fast}}\int_{0}^{T}\left(\sqrt{\Psi_{r}^{+}(\x^{\eps})}-\sqrt{\Psi_{r}^{-}(\x^{\eps})}\right)^{2}\d t\\
\leq& \sum_{r\in\mathcal{R}_\fast}\int_{0}^{T}S(J_{r}^{\eps}|\Psi^{+}_{r,\eps}(\x^{\eps}),\Psi^{-}_{r,\eps}(\x^{\eps}))\d t,
\end{align*}
or, equivalently, that 
\begin{align*}
&
\frac{1}{\eps}\sum_{r\in{\cal R}_{\fast}}\int_{0}^{T}\left(\sqrt{\Psi_{r}^{+}(\x^{\eps})}-\sqrt{\Psi_{r}^{-}(\x^{\eps})}\right)^{2}\d t\\
\leq& \sum_{r\in\mathcal{R}_\fast}\int_{0}^{T}S(J_{r}^{\eps}|\Psi^{+}_{r,\eps}(\x^{\eps}),\Psi^{-}_{r,\eps}(\x^{\eps}))\d t - \frac{1}{2}\int_{0}^{T}\sum_{r\in{\cal R}_{\fast}}J_{r}^{\eps}\gamma_{r}\cdot\log(\x^{\eps}/x_{s}^{*})\d t.
\end{align*}
We are now deriving an $\eps$-uniform bound for the right-hand side. Indeed, the bound on the first term follows from the a priori bound on the action functional \eqref{eq:AprioriCompactness}. For the second term, we
use the continuity equation, and get
\begin{align*}
-\int_{0}^{T}\sum_{r\in{\cal R}_{\fast}}J_{r}^{\eps}\gamma_{r}\cdot\log(\x^{\eps}/x_{s}^{*})\d t & =-\int_{0}^{T}\left(\dot{\x}^{\eps}-\sum_{r\in{\cal R}_{\slow}}J_{r}^{\eps}\gamma_{r}\right)\cdot\log(\x^{\eps}/x_{s}^{*})\d t\\
 & =-\int_{0}^{T}\dot{\x}^{\eps}\cdot\log(\x^{\eps}/x_{s}^{*})\d t+\int_{0}^{T}\sum_{r\in{\cal R}_{\slow}}J_{r}^{\eps}\gamma_{r}\cdot\log(\x^{\eps}/x_{s}^{*})\d t,
\end{align*}
which we both estimate as follows. We have
\[
\int_{0}^{T}\dot{\x}^{\eps}\cdot\log(\x^{\eps}/x_{s}^{*})\d t=\int_{0}^{T}\frac{\d}{\d t}{\cal H}(\x^{\eps}(t)|x_{s}^{*})\d t={\cal H}(\x^{\eps}(T)|x_{s}^{*})-{\cal H}(\x^{\eps}(0)|x_{s}^{*}),
\]
which is uniformly bounded in $\eps\to 0$ (also up to the boundary of $\sC$). Moreover, exploiting Lemma \ref{lem:EstimateCompactness}, we get the bound
\begin{align*}
&\int_{0}^{T}\sum_{r\in{\cal R}_{\slow}}J_{r}^{\eps}\gamma_{r}\cdot\log(\x^{\eps}/x_{s}^{*})\d t \\
\leq & \sum_{r\in\mathcal{R}_\slow}\int_0^T S(J_r^\eps|\Psi^+_r,\Psi^-_r) + k^+_r\bigg\{(\x^\eps)^{\gamma_r^-}(x_s^*)^{-\gamma_r}-(\x^\eps)^{\gamma_r^+} \bigg\} + k_r^- \bigg\{(\x^\eps)^{\gamma_r^+}(x_s^*)^{\gamma_r}-(\x^\eps)^{\gamma_r^-}\bigg\} \d t,
\end{align*}
whose right-hand side is again bounded by the a priori estimate \eqref{eq:AprioriCompactness} and the uniform bound $|\x^\eps(t)|\leq C$.\\
Hence, there exists a constant $\tilde C>0$ such that
\[
\frac{1}{\eps}\sum_{r\in{\cal R}_{\fast}}\int_{0}^{T}\left(\sqrt{\Psi_{r}^{+}(\x^{\eps})}-\sqrt{\Psi_{r}^{-}(\x^{\eps})}\right)^{2}\d t\leq \tilde C.
\]
Since the integrand
\[
F_\fast(\x^\eps(t)):=\sum_{r\in{\cal R}_{\fast}} \left(\sqrt{\Psi_{r}^{+}(\x^{\eps}(t))}-\sqrt{\Psi_{r}^{-}(\x^{\eps}(t))}\right)^{2}
\]
is non-negative, in the following denoted by $f_{\eps}=F_{\fast}(\x^{\eps}(t))$,
we conclude that the bounded sequence $f_{\eps}$ converges (up to
relabling) to zero a.e. in $[0,T]$. Since $F_{\fast}$ is continuous,
we conclude that $F_{\fast}(\x^{*})=0$ a.e., which means that   $\x^{*}(t)\in\mathcal{E}_\fast$
a.e. in $t\in[0,T]$.\\
{\bf Step 5:}    To get full compactness and the desired continuity for $\x^*$, we now exploit \eqref{eq:UFEC} and the reconstruction map $\sR$. From Proposition \ref{prop:PropertiesReconstructionAndProjection}, we know that the set of fast equilibria is characterized by the map $\sR:\sQ\to\cE_\fast$. Hence by continuity of $\sR$,
we get that $\x^{*}(t)=\sR(\sq^*(t))=\sR(Q_{\fast}\x^{*}(t))$ for a.e.
  $t\in[0,T]$. In particular, by Step 3, this also proves continuity of the limit trajectory $\x^*$,
and moreover, also strong convergence $\x^{\eps}\to \x^{*}$ in any
$\L^{p}([0,T],\R^{I})$, $p\in[1,\infty[$ by dominated convergence.
\end{proof}

\begin{prop}[Liminf estimate]\label{prop:liminf}
 Let $(\x^{\eps})_{\eps>0}$ be any sequence such that $\x^{\eps}\rightharpoonup \x^{*}$
in $\L^{1}([0,T],\sC)$ and $\sup_{\eps>0}\|\x_\eps(t)\|_{\L^\infty([0,T])}\leq C$ for some $C>0$. Then, we have the estimate $\liminf_{\eps\to0}{\cal L}_{\eps}(\x^{\eps})\geq{\cal L}_{\eff}(\x^{*})$.
\end{prop}

\begin{proof}
W.l.o.g we may assume that the sequence $(\x_\eps)_{\eps>0}$ satisfies  ${\cal L}_{\eps}(\x^{\eps})\leq C$, since
otherwise the claim is trivial. By the compactness result Proposition
\ref{prop:Compactness}, we conclude that the limit satisfies $\x^{*}\in\cE_\fast$
for a.e. $t\in[0,T]$. Moreover, there are fluxes $J^{\eps}=\left(J_{r}^{\eps}\right)_{r\in{\cal R}}$
such that the continuity equation $\dot{\x}^{\eps}=\WM^T J^\eps$
is satisfied, and we have 
\begin{align*}
{\cal L}_{\eps}(\x^{\eps})+\eps & \geq\int_{0}^{T}\sum_{r\in{\cal R}_{\slow}}S(J_{r}^{\eps}|\Psi_{r}^{+},\Psi_{r}^{-})\d t+\int_{0}^{T}\sum_{r\in{\cal R}_{\fast}}S(J_{r}^{\eps}|\frac{\Psi_{r}^{+}}{\eps},\frac{\Psi_{r}^{-}}{\eps})\d t\\
 & \geq\int_{0}^{T}\sum_{r\in{\cal R}_{\slow}}S(J_{r}^{\eps}|\Psi_{r}^{+},\Psi_{r}^{-})\d t.
\end{align*}
Recall by \eqref{eq:CompactnessFluxes}, that for the slow fluxes
we have weak convergence of $J_{r}^{\eps}\rightharpoonup J_{r}^{*}$
in $\L^{1}([0,T])$. Moreover, $\Psi_{r}^{\pm}$ are bounded  and  continuous, and we have
strong convergence $\x^{\eps}\to \x^{*}$ in $\L^{p}([0,T],\sC)$. 
Also note that the function $(J,x)\mapsto S(J|\Psi{{}^+}(x),\Psi^{-}(x))$ is lower
semicontinuous as a Legendre transform. Hence, we get the liminf estimate
by Ioffe's theorem (see e.g.    \cite[Th. 7.5]{fonleo07}  ),
i.e.
\[
\liminf_{\eps\to 0}{\cal L}_{\eps}(\x^{\eps})\geq\int_{0}^{T}\sum_{r\in{\cal R}_{\slow}}S(J_{r}^{*}|\Psi_{r}^{+},\Psi_{r}^{-})\d t.
\]
Applying the coarse-graining map $Q_{\fast}$ to the continuity equation
and taking the limit in the continuity equation, we get that 
\begin{align}\label{eq:Coarse-grainedCE}
Q_{\fast}\dot{\x}^{*}=Q_{\fast}\WM^TJ^* = \sum_{r\in\mathcal{R}_\slow} J_r^* Q_\fast \gamma_r.    
\end{align}
Hence, we can contract over all these fluxes and obtain the desired
liminf-estimate involving $\mathcal{L}_\eff$.
\end{proof}

   Next, we prove the limsup-inequality for the $\Gamma$-convergence result. To handle also paths that might reach the boundary $\partial\sC$, we use the positivity assumption \eqref{eq:PositivityR}. Note that the theory for HJE in Section \ref{sec5} assumes that the underlying   domain $\Omega$ has   a positive distance to the boundary, so the assumption \eqref{eq:PositivityR} is indeed not necessary for our main theorem.   

\begin{prop}[Construction of recovery sequence]\label{prop:ConstructionRecoverySequence}
Let $\x^{*}\in\L^{1}([0,T],\sC)$ with $\x^*(t)\leq C$ for a.e. $t\in[0,T]$. {   Let assumption \eqref{eq:PositivityR} holds.}  Then, there exists a family $\left(\x^{\eps}\right)_{\eps>0}$
such that $\x^{\eps}\to \x^{*}$  in $C([0,T],\sC)$   (and hence also in $\L^1([0,T],\sC)$)   and
$\lim_{\eps\to0}{\cal L}_{\eps}(\x^{\eps})={\cal L}_{\eff}(\x^{*})$.
\end{prop}

\begin{proof}
The proof is done in several steps. \\
{\bf Step 0:} First, we may assume that ${\cal L}_{\eff}(\x^{*})<\infty$, because
if ${\cal L}_{\eff}(\x^{*})=\infty$ we take the constant sequence
$\x^{\eps}=\x^{*}$ and the liminf provides that also ${\cal L}_{\eps}(\x^{\eps})\to\infty$.
Hence, we may assume that $\x^{*}\in \mathrm{AC}_x([0,T],\sC)$, $\sq:=Q_{\fast}\x^{*}\in\W^{1,1}([0,T],\R^{m_{\fast}})$
and that $\x^{*}(t)\in\cE_\fast$ for all $t\in[0,T]$.    In particular, we have that $\x^*$ is bounded.  Moreover,   the bound on the action functional provides the existence of general fluxes $\sJ_r\in \L^1([0,T])$  such that the coarse-grained continuity equation of the form $\dot{\sq}=\sum_{r\in{\cal R}_{\slow}}\sJ_{r}Q_{\fast}\gamma_{r}$ (see \eqref{eq:Coarse-grainedCE}) is satisfied.    \\
The aim is to construct a continuity equation for the whole system
$\dot{\x}^{\eps}=\sum_{r}J_{r}\gamma_{r}$, i.e. to find fluxes $J=(J_{r})_{r}$
and show that this sequence provides ${\cal L}_{\eps}(\x^{\eps})\to{\cal L}_{\eff}(\x^{*})$.    We do this in two steps: First, we handle the problems arising through the boundary  of $\sC$ by shifting the concentrations; secondly, we show that for positive concentrations the constant sequence provides a recovery sequence.    \\
{\bf Step 1:}    We first show that for a suitable shift $\x^\delta\geq 0$ with $\x^\delta\to\x^*$ we have for all slow reactions $r\in\mathcal{R}_\slow$ the limsup-inequality
\[
\limsup_{\delta\to0}\int_0^TS(J^r|\Psi^+_r(\x^\delta(t)),\Psi^+_r(\x^\delta(t)))\d t \leq\int_0^TS(J^r|\Psi^+_r(\x(t)),\Psi^+_r(\x(t)))\d t.
\]
In order to do so, we have to ensure, that $\x^\delta\in\mathcal{E}_\fast$ and that the continuity equation $\dot{\sq}=\sum_{r\in{\cal R}_{\slow}}\sJ_{r}Q_{\fast}\gamma_{r} $ is satisfied. For this, we use the positivity assumption \eqref{eq:PositivityR}. Hence, there exists $\bar \sq$ such that $\sR(\sq+\delta \bar \sq)_i>0$ and that $\sR(\sq+\delta\bar\sq)_i\geq \sR(\sq)_i$ on $[0,T]$. We define $\x^\delta:=\sR(\sq+\delta\bar\sq)$ and we have that $\x^\delta\to \x^*$ monotonously. Note that the fluxes are not changed, because the derivative of $\sq$ is independent of the shift.\\
From the monotone convergence $\x^\delta\to \x^*$ the desired convergence follows by a standard dualization argument regardless the degeneracy of $S$ at the boundary $\partial \sC$. Indeed, introducing $g^\delta =\Psi^+_r(\x^\delta)$, $h^\delta =\Psi^-_r(\x^\delta)$, which also converges montonously to $g$ and $h$ on $[0,T]$, respectively, we want to show 
\[
g^\delta\to g, h^\delta\to h, S(J^r|g,h)\in\L^1([0,T])\quad \Rightarrow \quad \limsup_{\delta\to 0}\int_0^T S(J|g^\delta,h^\delta) \d t
\leq \int_0^T S(J|g,h) \d t . \]
For this, we use the rewriting $S(J|g,h)=\inf_{J=u-w} \mathcal{H}(u|g) + \mathcal{H}(w|h)$ from Lemma \ref{prop:PropertiesS}. 
In particular, we may assume that for all $n\in\N $ we have functions $u^n,w^n$ with $J=u^n-w^n$ such that 
\begin{equation}\label{hghg}
\int_0^T S(J|g,h) \d t \geq  \int_0^T \mathcal{H}(u^n|g) + \mathcal{H}(w^n|h) \d t - \frac 1 n.
\end{equation}
Observe that for fixed $n\in\N$, we have by monotone convergence that $\int_0^T\mathcal{H}(u^n|g^\delta)\d t \to \int_0^T\mathcal{H}(u^n|g)\d t$ (and similarly also for $\int_0^T \mathcal{H}(w^n|h^\delta) \d t$). Adding these two limits and using \eqref{hghg}, we obtain
\[
\int_0^T S(J|g,h) \d t \geq \int_0^T \mathcal{H}(u^n|g^\delta) + \mathcal{H}(w^n|h^\delta) + o(\delta) - \frac 1 n \geq \int_0^TS(J|g^\delta,h^\delta) + o(\delta) - \frac 1 n.
\]
This finishes the proof of the claim because $\delta>0$ and $n\in\N$ are arbitrary. 
   
{\bf Step 2:} By Step 1, we may assume that $\x^*$ is strictly positive. We define the constant sequence $\x^{\eps}=\sR(\sq)=\sR(Q_{\fast}\x^{*})$.
Then, we have
\[
\frac{\d}{\d t}\x^{\eps}=\D\sR(\sq)\dot{\sq}=\sum_{r\in{\cal R}_{\slow}}\sJ_{r}\left(\D\sR(\sq)Q_{\fast}\right)\gamma_{r}.
\]
By the definition of the projection $\mathbb{P}$ in Proposition \ref{prop:PropertiesReconstructionAndProjection}, this can be rewritten
as 
\[
\dot{\x}^{\eps}=\sum_{r\in{\cal R}_{\slow}}\sJ_{r}\left(I-\mathbb{P}(\x(t))\right)\gamma_{r}=\sum_{r\in{\cal R}_{\slow}}\sJ_{r}\gamma_{r}+\sum_{r\in{\cal R}_{\slow}}\sJ_{r}\mathbb{P}(\x(t))\gamma_{r}.
\]
  For the last term, we use that $\mathrm{range}(\mathbb{P})=\Gamma_{\fast}$. Hence, there are (time-dependent) coefficients $(J_r=J_r(t))_{r\in\mathcal{R}_\fast}$ such
that $\sum_{r\in{\cal R}_{\slow}}\sJ_{r}\mathbb{P}(\x(t))\gamma_{r}=\sum_{r\in{\cal R}_{\fast}}J_{r}(t)\gamma_{r}$.
The map $\sJ_{r}\mapsto J_{r}$ is linear and uniformly bounded by Proposition \ref{prop:PropertiesReconstructionAndProjection}. Hence, defining the
slow fluxes $J_{r}:=\sJ_{r}$ for $r\in\mathcal{R}_\slow$, we have the exact continuity equation involving all reaction $r$
\[
\dot{\x}^{\eps}=\sum_{r\in\mathcal{R}_\slow\cup\mathcal{R}_\fast}J_{r}\gamma_{r}.
\]
By the compactness argument for slow fluxes in \eqref{eq:FluxesInLC},
we may assume that the slow fluxes $\sJ_{r}$, $r\in{\cal R}_{\slow}$
are in $\L^{{\cal C}}([0,T])$. Since the fast fluxes $J_{r}$, $r\in{\cal R}_{\fast}$
are linearly dependent of $\sJ_{r}$ and obtained by a bounded map, they are also in $\L^{\cC}([0,T])$. \\
{\bf Step 3:} With the above construction, the $\eps$-dependent Lagrangian is given by
\[
L_{\eps}(x,v)=\inf_{v=\WM^T J} \left\{ \sum_{r\in{\cal R}_{\slow}}S(J_{r}|\Psi_{r}^{+}(x),\Psi_{r}^{-}(x))+\sum_{r\in{\cal R}_{\fast}}S(J_{r}|\tfrac{1}{\eps}\Psi_{r}^{+}(x),\tfrac{1}{\eps}\Psi_{r}^{-}(x))\right\} ,
\]
and it suffices to show that for all fast reactions $r\in{\cal R}_{\fast}$, 
we have 
\[
\int_{0}^{T}S(J_{r}|\frac{1}{\eps}\Psi_{r}^{+}(\x),\frac{1}{\eps}\Psi_{r}^{-}(\x))\d t=\int_{0}^{T}\frac{\Psi_{r}(\x^{*})}{\eps}\cC(\frac{\eps J_{r}}{\Psi_{r}(\x^{*})})\d t\to0.
\]
Here we have used that on the slow manifold we have $\Psi_{r}^{+}=\Psi^-_{r}=:\Psi_{r}$. 
Indeed, to see this convergence holds true, we use the estimate
\[
\frac{1}{2}|r|\log(|r|+1)\leq\cC(r)\leq2|r|\log(|r|+1).
\]
Denoting $\Psi_{r}(x^{*})=\rho$ with $0<\frac{1}{M}\leq\rho\leq M<\infty$,
we have for small $\eps>0$ that
\begin{align*}
\frac{\rho}{\eps}\cC(\frac{\eps J_{r}(t)}{\rho}) & \leq\frac{\rho}{\eps}\left\{ 2|\frac{\eps J_{r}(t)}{\rho}|\log\left(|\frac{\eps J_{r}(t)}{\rho}|+1\right)\right\} =2|J_{r}(t)|\log\left(|\frac{\eps J_{r}(t)}{\rho}|+1\right)\\
 & \leq2|J_{r}(t)|\log\left(|J_{r}(t)|+1\right)\leq4\cC(J_{r}(t)),
\end{align*}
which is integrable. Moreover, we have that $\log\left(|\frac{\eps J_{r}(t)}{\rho}|+1\right)\leq\frac{\eps}{\rho}|J_{r}(t)|$,
which is also integrable and converges pointwise to 0 for almost all
$t\in[0,T]$. Hence, also $|J_{r}(t)|\log\left(|\frac{\eps J_{r}(t)}{\rho}|+1\right)$
converges pointwisely to zero for almost all
$t\in[0,T]$, which implies, by the dominated convergence theorem,
that also $$\int_{0}^{T}\frac{\Psi_{r}(\x^{*})}{\eps}\cC(\frac{\eps J_{r}(t)}{\Psi_{r}(x^{*})})\d t\to0.$$
Hence, we have that the constant sequence satisfies ${\cal L}_{\eps}(\x^{*})\to{\cal L}_{\eff}(\x^{*})$,
which proves the claim. 

\end{proof}

Putting the above results together,  Theorem \ref{thm:GammaConvergence} can now be proved.
\begin{proof}[Proof of Theorem \ref{thm:GammaConvergence}]
    The claim now follows from Proposition \ref{prop:liminf} and Proposition \ref{prop:ConstructionRecoverySequence}.
\end{proof}

\section{Convergence of viscosity solutions of $\eps$-dependent HJE}\label{sec5}

  In this section, we   fix any bounded domain $\Omega$, $\overline{\Omega} \subset \bR^I_+$ such that the boundary of $\Omega$ has a positive distance to the boundary $\partial \R^I_+$ and  such that   \eqref{eq:FastSlowHJE} holds. Notice for $\overline{\Omega}$, we    have the lower estimate
  \begin{equation}\label{eq:lb}
      \exists c_0>0,\,  \forall x\in\overline{\Omega},\, \forall r\in\mathcal{R}:\quad  \Psi_r^+(x), \Psi_r^-(x) \geq c_0.
  \end{equation}   
  Recall the
 Hamiltonian $H_\eps(x,p)$ from \eqref{eq:Heps} on $\Omega \times \bR^I$.
Recalling  the stoichiometric subspace $\Gamma=\mathrm{span}\left\{ \gamma_{r}:r\in{\cal R}\right\}$ and $\Gamma^{\perp}=\left\{ q\in\R^{I}:\forall\gamma\in\Gamma:q\cdot\gamma=0\right\}$, it is easy to see for any $p\in \bR^I$, we have
that 
\begin{equation}\label{gg}
H_\eps(x,p)=H_\eps(x, p + p^\perp), \quad  \forall  p^\perp\in \Gamma^{\perp}.
\end{equation}
In particular $H_\eps$ is degenerate because for $p\in \Gamma^{\perp}$ we have $H_\eps(x,p)\equiv \pt_x H_\eps(x,p)\equiv 0.$
Note that we have already seen this degeneracy for the corresponding Lagrangian $L_\eps$ because
if $v\neq \Gamma$ then $L_\eps(x,v)=+\8$.   
 From the above degeneracy of $H_\eps$, one can expect that the usual assumptions on Hamiltonian for the wellposedness of Hamilton-Jacobi equation will only be effective for $p$   evaluated (or measured) in the direction of $\gamma_r$, i.e. for the quantities $\gamma_r \cdot p $, which is shown in the next lemma.  
 
\begin{lem}\label{lem_coe}
For any $\eps>0$, we have the uniform coercivity of $H_\eps$   on $\Omega\subset \R^I$ and in the directions of  $\Gamma \subset \R^I$, i.e., 
\begin{equation}
\inf_{x\in \Omega} H_\eps(x,p) \to +\8 \quad  \text{ whenever } |\gamma_r \cdot p| \to +\8 \text{ for some } r\in \cal{R}.
\end{equation}
\end{lem}
\begin{proof}
From the definition of $H_\eps$ and \eqref{eq:lb}, we have
\begin{align*}
H_\eps(x,p)=& \sum_{r\in \cal{R}}      \bbs{ \Psi^+_r(x)\e^{ \gamma_r  \cdot  p}   -  \Psi^+_r(x) +  \Psi^-_r(x)\e^{- \gamma_r  \cdot p }   -  \Psi^-_r(x) }\\
\geq & \sum_{r\in \cal{R}_\slow} [\min\{\Psi_r^+(x), \Psi_r^-(x)\} \bbs{\e^{\gamma_r \cdot p} + \e^{-\gamma_r \cdot p}} - \Psi_r^+(x)- \Psi_r^-(x)]\\
& +\frac{1}{\eps}\sum_{r\in \cal{R}_\fast} [\min\{\Psi_r^+(x), \Psi_r^-(x)\} \bbs{\e^{\gamma_r \cdot p} + \e^{-\gamma_r \cdot p}} - \Psi_r^+(x)- \Psi_r^-(x)]
\\
\geq & \sum_{r\in \cal{R}_\slow} [ c_0\bbs{\e^{\gamma_r \cdot p} + \e^{-\gamma_r \cdot p}} - \Psi_r^+(x)- \Psi_r^-(x)]\\
& +\frac{1}{\eps}\sum_{r\in \cal{R}_\fast} [c_0 \bbs{\e^{\gamma_r \cdot p} +\e^{-\gamma_r \cdot p}} - \Psi_r^+(x)- \Psi_r^-(x)]
\\
\geq & \sum_{r\in \cal{R}_\slow} [2c_0|\gamma_r\cdot p| - \Psi_r^+(x)- \Psi_r^-(x)] +\frac{1}{\eps}\sum_{r\in \cal{R}_\fast} [2c_0|\gamma_r\cdot p| - \Psi_r^+(x)- \Psi_r^-(x)] \to +\8 
\end{align*}
uniformly in $x\in\Omega$, whenever $|\gamma_r \cdot p| \to +\8$   for some $r\in \cal{R}$.
\end{proof}

Recall the fast-slow HJE problem for the Hamiltonian $H_\eps$ in $\Omega$ with state constraints is given in \eqref{eq:FastSlowHJE}, i.e., 
\begin{equation*} 
\begin{aligned}
\pt_t u^\eps(x,t) + H_\eps(x, \nabla u^\eps(x))\leq 0, \quad (x,t)\in \Omega\times(0,T),\\
\pt_t u^\eps(x,t) + H_\eps(x, \nabla u^\eps(x))\geq 0, \quad (x,t)\in \overline{\Omega}\times(0,T),\\
u^\eps(x,0) = u_0^\eps(x), \quad x\in \Omega.
\end{aligned}
\end{equation*}
The specific assumptions on initial data $u_0^\eps$   will be given below in Assumption \ref{assumption}. Here we also recall the definition of viscosity solutions to \eqref{eq:FastSlowHJE} \cite{capuzzo1988hamilton}.
\begin{defn}\label{def:vis}
    Given any $\eps>0$, a function $u^\eps$ is a viscosity solution to \eqref{eq:FastSlowHJE} if it is both a viscosity subsolution and supersolution in the following sense: $u^\eps$ is a viscosity subsolution of 
  $$\pt_t u^\eps +  H_\eps(x, \nabla u^\eps(x,t)) \leq 0$$
  on  $\Omega\times(0,T)$, if for every $\phi \in C^1(\Omega\times(0,T))$   and every $(x_0,t_0)\in \Omega\times(0,T)$ such that $u-\phi$ has a maximum at $(x_0, t_0)$, we have
  $$\pt_t \phi(x_0, t_0) + H_\eps(x_0, \nabla\phi(x_0,t_0))\leq 0;$$
   $u^\eps$ is a viscosity supersolution of 
  $$\pt_t u^\eps +  H_\eps(x, \nabla u^\eps(x,t)) \geq 0$$
  on  $\overline{\Omega}\times(0,T)$, if for every $\phi \in C^1(\overline{\Omega}\times(0,T))$   and every $(x_0,t_0)\in \overline{\Omega}\times(0,T)$ such that $u-\phi$ has a minimum at $(x_0, t_0)$, we have
  $$\pt_t \phi(x_0, t_0) + H_\eps(x_0, \nabla\phi(x_0,t_0)))\geq 0.$$
\end{defn}

 For fixed $\eps>0$,   the variational representation for the state-constraint viscosity solution to \eqref{eq:FastSlowHJE} follows by \cite{mitake2008asymptotic} and is given by
\begin{equation}\label{eq:RepresentationForEps}
u^\eps(x,t)=\inf_{\mathrm{x}\in \mathrm{AC}([0,t]; \overline{\Omega}),\, \mathrm{x}(t)=x} \big(u^\eps_0(\mathrm{x}(0)) +  
   \int_0^t L_\eps(\mathrm{x}_s,\dot{\mathrm{x}}_s)\ud s\big). 
\end{equation}
{We remark that although our Hamiltonian $H_\eps$ is degenerate  and indeed coercive in $\Gamma$ (see Lemma \ref{lem_coe}), this representation  still holds.}
However, the  existence and representation results are   not uniform in $\eps$.
We also  note that for fixed $\eps>0$, we have
$H_\eps\in \BUC(\Omega\times B_R)$ for any $R>0$. Here $\BUC(\Omega\times B_R)$ means bounded and uniformly continuous functions on $\Omega\times B_R$. 
However, the fast-slow Hamiltonian $H_\eps$ is not uniformly bounded when $\eps \to 0$.
 
We introduce a smaller function space 
\begin{equation}
   \Pi: =   \{u\in \BUC(\Omega): \,\,\nabla u\in\Gamma^\perp_\fast\} =\{u\in \BUC(\Omega): \,\,  \gamma_r\cdot  \nabla u =0, \forall r\in\mathcal{R}_\fast\}    
\end{equation}
and we propose the following conditions for the initial data:

\begin{asmp}[Well prepared initial data] \label{assumption} The family of initial data $(u_0^\eps)_{\eps>0}$ of   \eqref{eq:FastSlowHJE} satisfies the following   conditions 
\begin{enumerate}[(i)]
    \item $\forall \eps>0, \,\, u_0^\eps(\cdot) \in \Pi;$
    \item there exists $C>0$ such that $\|u_0^\eps\|_{C^1(\overline{\Omega})} \leq C$ uniformly in $\eps>0;$
    \item there exists $u^*_0\in C^1(\overline{\Omega})$ such that $u^\eps_0\to u^*_0 \text{  in } C(\overline{\Omega}).$
\end{enumerate}    
\end{asmp}

With the above assumption, we can derive $\eps$-uniform Lipschitz continuity of the solution.

\begin{prop}\label{prop_Lip}
Under Assumption \ref{assumption}, let $u^\eps\in \BUC(\overline{\Omega}\times [0,T])$ be a viscosity solution to   \eqref{eq:FastSlowHJE}, then we have the uniform Lipschitz estimates for $u^\eps$, i.e., there exists a constant $C_L>0$ such that
\begin{equation}\label{Lip}
|\pt_t u^\eps(x,t)|+|\gamma_r\cdot\nabla u^\eps(x,t)| \leq C_L, \quad \text{ for a.e. }\, \, (x,t)\in \Omega\times[0,T],  \,\, \forall r\in \cal{R}.
\end{equation}
\end{prop}
\begin{proof}
The proof is performed in three steps.\\
{\bf Step 1:} Existence of $u^\eps$:\\
     One can directly verify that $u^\eps\in \BUC(\overline{\Omega}\times[0,T])$ given by \eqref{eq:RepresentationForEps} is  a viscosity solution  to \eqref{eq:FastSlowHJE} for fixed $\eps$. It indeed can be proved via    Perron's method  \cite[Theorems 5.2, 5.8]{mitake2008asymptotic}. \\
{\bf Step 2:} $u^\eps$ is Lipschitz continuous in time for all $t\in [0,T]$ uniformly in $\eps$:\\
        Since $\|\nabla u_0^\eps\|_{\L^\8}\leq C$ and $\gamma_r \cdot \nabla u_0^\eps = 0, \,\, \forall r \in \cal{R}_\fast$, we conclude that there exists some constant $C^*>0$    such that
\begin{equation}
|H_\eps(x,\nabla u_0^\eps(x))|_{\L^\8} \leq C^* \quad \text{ uniformly in } \eps.
\end{equation}
Thus one can verify that
\begin{equation}\label{varphi}
\begin{aligned}
\underline{\varphi}(x,t) := u_0^\eps - C^* t \quad \text{ is a classical subsolution};\\
\overline{\varphi}(x,t) := u_0^\eps + C^* t \quad \text{ is a classical supersolution}.
\end{aligned}
\end{equation}
Then based on the comparison principle \cite[Theorem 3.5]{mitake2008asymptotic},   we have
\begin{equation}
\underline{\varphi}(x,t) \leq u^\eps(x,t) \leq \overline{\varphi}(x,t)  \quad \text{ uniformly in } \eps.
\end{equation}
In particular, together with Assumption   \ref{assumption} for $u_0^\eps$, we have that $u^\eps$ is bounded uniformly in $\eps>0$ and $(x,t)\in \overline{\Omega}\times[0,T]$. 
Then we have
\begin{equation}
-C^* \leq \frac{u^\eps(x,t)-u_0^\eps}{t} \leq C^*.
\end{equation}
Taking the supremum for $t\geq 0$, we obtain the Lipschitz continuity in time at $t=0$. {Moreover, to obtain Lipschitz continuity in time at any $t$, we note  the Hamiltonian is time homogeneous and thus any time translation of a solution is still a solution $v(x,t)=u^\eps(x,s+t)$ with initial data $v_0=v(x,0)=u^\eps(x,s)$. From
the comparison principle and 
$$v_0-\|u_0^\eps-v_0\|_{\L^\8}\leq u_0^\eps \leq v_0+\|u_0^\eps-v_0\|_{\L^\8},$$
we have
$$v(x,t)-\|u_0^\eps-v_0\|_{\L^\8}\leq u^\eps(x,t) \leq v(x,t)+\|u_0^\eps-v_0\|_{\L^\8}.$$
This means 
$$u^\eps(x,s+t)-\|u_0^\eps-v_0\|_{\L^\8}\leq u^\eps(x,t) \leq u^\eps(x,s+t)+\|u_0^\eps-v_0\|_{\L^\8}$$
and thus
\begin{equation}
    \left| \frac{u^\eps(x,t+s)-u^\eps(x,t)}{s} \right|\leq \left\| \frac{u^\eps(\cdot,s)-u^\eps(\cdot,0)}{s} \right\|_{\L^\8} \leq C^*.
\end{equation}
 Taking the supremum for $s\geq 0$, one   concludes} the Lipschitz continuity in time at any $t$
\begin{equation}\label{Lip_t}
|\pt_t u^\eps(x,t)|\leq C^* \quad \text{ for a.e. } (x,t)\in \Omega\times[0, T].
\end{equation}
{\bf Step 3:} $u^\eps$ is Lipschitz continuous in space uniformly in $\eps$:


First fix any $x\in \Omega$ and $t\in(0,+\8)$. Given any $R>0$ such that $\overline{B_R}(x)\subset \overline{\Omega}$, for any $y\in (x+\Gamma) \cap \overline{B_R}(x)$, consider the function
\begin{equation}
\varphi(y,s) = u^\eps(y,s) - M|s-t| - C_L|y-x| - \frac{\delta}{R^2-|x-y|^2},
\end{equation}
where $C_L, M>0$  are constants large enough to be chosen later and $\delta>0$ is an arbitrary   constant.
Then due to the boundedness of $u^\eps$ and $\varphi \to -\8$ as $y\to \pt B_R(x),$   
there exists {a maximizer   $y^* \in (x+\Gamma)\cap B_R(x)$  and $t^*\in (0,+\8)$ of $\varphi(y,s)$}, such that
\begin{equation}\label{varphiM}
\begin{aligned}
&\varphi(y^*,t^*)= u^\eps(y^*,t^*) - M|t^*-t| - C_L|y^*-x| - \frac{\delta}{R^2-|x-y^*|^2} \\
\geq  &\varphi(y,s) = u^\eps(y,s) - M|s-t| - C_L|y-x| - \frac{\delta}{R^2-|x-y|^2}
\end{aligned}
\end{equation}
for any $y\in (x+\Gamma)\cap \overline{B_R(x)}$ and any $s\in(0,+\8)$.
Since $u$ is Lipschitz in time, for sufficient large $M>0$, one must have $t^*=t$. Otherwise, one can take $y=y^*$ and $s=t$ in \eqref{varphiM} to obtain
 $$ u^\eps(y^*,t^*) - M|t^*-t|\geq u^\eps(y^*,t).$$
This, together with   the Lipschitz constant $C^*$ in time in \eqref{Lip_t}, we further have
$$u^\eps(y^*,t) + (C^*- M)|t^*-t| > u^\eps(y^*,t^*) - M|t^*-t|\geq u^\eps(y^*,t),$$
which is impossible for   $M>C^*.$ Hence one must have $t^*=t.$

Now we show that $y^*=x$ for $C_L$ sufficiently large. To show this by the arguments of contradiction, we first assume $y^*\neq x$.
From the Lipschitz estimate in time \eqref{Lip_t}, and because $u^\eps$ is a viscosity subsolution and $|y-x|$ is differentiable at any $y\neq x$, we know
\begin{equation}\label{original}
H_\eps(y^*, C_L \frac{y^*-x}{|y^*-x|}+ \frac{2\delta( y^*-x)}{(R^2-|x-y^*|^2)^2})\leq C,
\end{equation}
where $C$ is a constant depending only on $C^*$.
Then using the same estimate of Lemma \ref{lem_coe}, we have
\begin{equation}\label{tm_HC}
\begin{aligned}
C\geq & H_\eps(y^*, C_L \frac{y^*-x}{|y^*-x|}+ \frac{2\delta( y^*-x)}{(R^2-|x-y^*|^2)^2})\\
\geq & \sum_{r\in \cal R_\slow} \Big[2c_0 \big(C_L\Big|\gamma_r\cdot \frac{y^*-x}{|y^*-x|}\Big|-  \Big|\frac{2\delta\gamma_r\cdot(x-y^*)}{(R^2-|x-y^*|^2)^2}\Big|\big) - \Psi_r^+(x)- \Psi_r^-(x)\Big] \\
&+\frac{1}{\eps}\sum_{r\in \cal R_\fast} \Big[2c_0\big(C_L\Big|\gamma_r\cdot \frac{y^*-x}{|y^*-x|}\Big|-  \Big|\frac{2\delta\gamma_r\cdot(x-y^*)}{(R^2-|x-y^*|^2)^2}\Big|\big) - \Psi_r^+(x)- \Psi_r^-(x)\Big].
\end{aligned}
\end{equation}
Since $\Psi^{\pm}_r$ is bounded and $y^*$ is in the interior of $B_R(x)$, we have
\begin{equation}
\sum_{r\in \cal R_\fast} \Big[2c_0 C_L\Big|\gamma_r\cdot \frac{y^*-x}{|y^*-x|}\Big| - \Psi_r^+(x)- \Psi_r^-(x)\Big] \leq \eps C.
\end{equation}
Taking $C_L$ sufficiently large, this implies 
\begin{equation}\label{If}
\gamma_r\cdot  (y^*-x)=0, \quad \forall r\in \cal{R}_\fast.
\end{equation}
Therefore, plugging \eqref{If} into the original Hamiltonian \eqref{original}, the summation for the fast reaction vanish, so \eqref{original} implies
 \begin{equation}
 \sum_{r\in \cal{R}_\slow} \Big[2c_0 \big(C_L\big|\gamma_r\cdot \frac{y^*-x}{|y^*-x|}\big|-  \big|\frac{2\delta\gamma_r\cdot(x-y^*)}{(R^2-|x-y^*|^2)^2}\big|\big) - \Psi_r^+(x)- \Psi_r^-(x)\Big]\leq C.
 \end{equation}
With the same arguments as \eqref{If}, we obtain
 \begin{equation}\label{Ir}
\gamma_r\cdot (y^*-x)=0, \quad \forall r\in \cal{R}_\slow.
\end{equation}
Combining \eqref{If} and \eqref{Ir}, notice also $y^*-x\in \Gamma$, thus we must have $y^*=x$.
Therefore, for such a sufficient large $C_L$,  \eqref{varphiM} yields 
\begin{equation}
u^\eps(y,s)-u^\eps(x,t) \leq C_L|y-x|+M|s-t|-\frac{\delta}{R^2} + \frac{\delta}{R^2-|x-y|^2}.
\end{equation}
Taking $\delta\to 0$, $s=t$ and then exchanging $x,y$  shows 
\begin{equation}
|u^\eps(y,t)-u^\eps(x,t)| \leq C_L|y-x|, \quad \text{ for a.e. }\, \, t\in  [0,T],  \,x, y\in      {\Omega} \, \text{ with } x-y \in \Gamma.
\end{equation}
This implies
\begin{equation}\label{Lip_x}
    |\gamma_r \cdot \nabla u^\eps(x,t)| \leq C_L, \quad \text{ for a.e. }\, \, (x,t)\in \Omega\times[0,T], \,\, \forall r\in \cal{R}.
\end{equation}
Combining \eqref{Lip_t} and \eqref{Lip_x}, we finish the proof. 
\end{proof}


\begin{rem}
We remark that the bound in \eqref{Lip} can be refined for the fast reactions.    Indeed, from the $\eps$-uniform bound
\begin{align*}
    \frac 1 \eps |\sum_{r\in{\cal R}_{\fast}}\Psi_{r}^{+}(x)\left(\e^{\gamma_{r}\cdot \nabla u^\eps}-1\right)+\Psi_{r}^{-}(x)\left(\e^{-\gamma_{r}\cdot \nabla u^\eps}-1\right)|\leq C,
\end{align*}
we conclude  for fast reactions $r\in{\cal R}_{\fast}$ and for $x\in \MS=\cE_\fast\cap\overline\Omega$ (such that we have $\Psi_{r}^{+}(x)=\Psi_{r}^{-}(x)$)   that
\begin{align*}
0\leq 2\Psi_r^+(x)(\cosh (\gamma_r \cdot \nabla u^\eps)-1)\leq \sum_{r\in{\cal R}_{\fast}} 2\Psi_r^+(x)(\cosh (\gamma_r \cdot \nabla u^\eps)-1) \leq C \eps  \to 0.
\end{align*}
This, together with \eqref{eq:lb}, yields that for $r\in{\cal R}_{\fast}$,
$$\gamma_r\cdot \nabla   u^\eps \to    0 \quad \text{ uniformly in } t\in [0,T], x\in \MS.$$
\end{rem}

Based on the above Lipschitz estimates for the viscosity solution, which is uniform in $\eps$, we have the following uniqueness theorem. With the Lipschitz estimates, the uniqueness is standard, however, we still give a brief proof since our Hamiltonian is  degenerate on $\Gamma^\perp$  and the Lipschitz estimate is in the sense of \eqref{Lip}.
\begin{thm}\label{thm:HJElimit}
Under Assumption \ref{assumption},      \eqref{eq:FastSlowHJE} has a   unique Lipschitz  viscosity solution $u^\eps(x,t)$, for which, uniform Lipschitz estimate \eqref{Lip} holds.
\end{thm}
\begin{proof}
{   The existence is obvious thanks to the variational representation \eqref{eq:RepresentationForEps}. }
Based on Lipschitz estimates in Proposition \ref{prop_Lip}, we can introduce a modified Hamiltonian
$\widetilde{H}_\eps\in C( \Omega\times \bR^n)$ such that
\begin{equation}
\widetilde{H}_\eps(x,p) =\left\{ \begin{array}{cc}
H_\eps(x,p), \quad & \text{ if } \,  |\gamma_r \cdot p|\leq C_L \text{ for all } r\in \cal R;\\
|p|, \quad & \text{ if } \,  |\gamma_r \cdot p|\geq 2C_L \text{ for some } r\in \cal R,
\end{array}
\right.
\end{equation}
and Lipschitz continuously connected otherwise.
We point out that although the modified Hamiltonian is based on the  spatial Lipschitz estimates  only in the directions of $\Gamma$, it still satisfies the usual assumptions for Hamiltonian (cf. \cite{bardi1997optimal, Tran21}): for any $p,q\in \bR^I$ and $x,y\in \Omega$:
\begin{equation}\label{H_asm}
\begin{aligned}
|\widetilde{H}_\eps(x,p)-\widetilde{H}_\eps(y,p)|\leq C_\eps(1+|p|)|x-y|,\quad 
|\widetilde{H}_\eps(x,p)-\widetilde{H}_\eps(x,q)|\leq C_\eps|p-q|.
\end{aligned}
\end{equation}
Indeed, whenever $|\gamma_r \cdot p|\leq C_L, \forall r\in\mathcal{R}$, we have $|\widetilde{H}_\eps(x,p)-\widetilde{H}_\eps(y,p)|\leq C_\eps|x-y|$; and otherwise $|\widetilde{H}_\eps(x,p)-\widetilde{H}_\eps(y,p)|\leq C_\eps(1+|p|)|x-y|$. Moreover, if $|\gamma_r \cdot p|\leq C_L, \forall r\in\mathcal{R}$, we have $|\widetilde{H}_\eps(x,p)-\widetilde{H}_\eps(x,q)|\leq C_\eps\sum_r|\gamma_r\cdot(p-q)|\leq C_\eps|p-q|$; and otherwise $|\widetilde{H}_\eps(x,p)-\widetilde{H}_\eps(x,q)|\leq C_\eps|p-q|$.

Thus for any  $u^\eps\in \BUC(\overline{\Omega}\times[0,T])$   solves 
  \eqref{eq:FastSlowHJE} with the Lipschitz estimate \eqref{Lip} will also solve the HJE
  with $H_\eps$ replaced by $\widetilde{H}_\eps$. Then by \eqref{H_asm}, one has the uniqueness of viscosity solution to \eqref{eq:FastSlowHJE}
  with $H_\eps$ replaced by $\widetilde{H}_\eps$,   so $u^\eps$ is the unique Lipschitz viscosity solution to \eqref{eq:FastSlowHJE}. 
\end{proof}

Using the $\eps$-uniform Lipschitz estimate and  the $\eps$-uniform boundedness, we then apply the Arzel\'a-Ascoli Theorem in the Banach space $C(K\times[0,T])$ for any compact subset $K\subset\Omega$.
\begin{cor}\label{cor:ExistenceLimitHJE}
 Let $u^\eps$ be the unique Lipschitz viscosity solution obtained in Theorem \ref{thm:HJElimit}. For any  compact subset $K\subset\Omega$, {   there exist $\hat{u}\in \BUC(\overline{\Omega}\times[0,T])$ and a subsequence $\eps_k$ such that 
 $u^{\eps_k}$ converges to $\hat{u}$ uniformly in $C(K\times [0,T])$ and the uniform Lipschitz estimate \eqref{Lip} holds for $\hat{u}$.}
\end{cor}
 {   We remark that the limit $\hat{u}$ for $x\in \MS=\cE_\fast\cap\overline\Omega$ is actually proven to be unique via the Gamma-convergence result and Proposition \ref{prop:identify}. Thus one indeed obtains the uniform convergence on $K\times[0,T],$ for any compact $K\subset\MS.$
 }
\bigskip

\section{Identification of limiting HJE with semigroup representation}\label{sec6}
The aim of the section is to identify the limit $u^*$ of the sequence of solutions $u^\eps$ from Corollary \ref{cor:ExistenceLimitHJE} by the $\Gamma$-convergence result from Section \ref{s:GammaConvergence}.    In particular, the following representation formula provides also a uniqueness result for the limit.      The viscosity solutions $u^\eps$ in the representation formula \eqref{eq:RepresentationForEps} is given as an infimum and  $\Gamma$-convergence is tailored to provide convergence of minimizers. In particular, we want to show that $u^*$ satisfies the variational representation \eqref{eq:RepresentationUstar} and solves the Hamilton-Jacobi equation for the effective $H_\eff$. Finally, combining Theorem \ref{thm:effectiveHJE} and the limit $u^*$ of the sequence of solutions $u^\eps$ from Corollary \ref{cor:ExistenceLimitHJE}, we finish the proof of the \textbf{Main Theorem} in the introduction. 

\subsection{Consequence of the $\Gamma$-convergence result}
Before we connect $u^*$ with a representation formula, we first recall how $\Gamma$-convergence provides the identification of the limit of  minimizers of functionals. For that, let
 $F_{\eps}\stackrel{\Gamma}{\to}F_{0}$ and let $\x_\eps$ be s.t.  $F_\eps(\x_{\eps})=\inf F_{\eps}$
and $\x_{\eps}\to \x_{0}$. Then we have that $\x_{0}=\inf F_{0}$. Indeed, we have by the $\Gamma$-liminf, that $\inf F_{0}\leq F_{0}(\x_{0})\leq\liminf_{\eps\to0}F_{\eps}(\x_{\eps})=\liminf_{\eps\to0}\inf F_{\eps}$.
Moreover, for any $\x$ there is a recovery sequence, $\bar{\x}_{\eps}\to \x$
such that $F_{0}(\x)=\limsup_{\eps\to0}F_{\eps}(\bar{\x}_{\eps})\geq\limsup_{\eps\to0}\inf F_{\eps}$.
Since, $\x$ is arbitrary, we obtain that $\inf F_{0}\geq\limsup_{\eps\to0}\inf F_{\eps}$
which thus concludes that $\lim_{\eps\to0}\inf F_{\eps}=\inf F_{0}=F_{0}(\x_{0})$.

   In our situation, we have on $\L^{1}([0,T],\sC)$ that ${\cal L}_{\eps}\xrightarrow{M}{\cal L}_{\eff}$ provided that the final point of the trajectories is fixed, see \eqref{thm:GammaConvergence}.    
For fixed $u_{0}^\eps\in C(\overline{\Omega})$, $T>0$ and $x\in\overline{\Omega}$, we recast the viscosity solution $u^\eps$, represented in \eqref{eq:RepresentationForEps}, as 
\[
u^{\eps}(x,T)=\inf\left\{ u_{0}^\eps(\x_{\eps}(0))+{\cal L}_{\eps}(\x_{\eps}):{ \x_{\eps}\in\mathrm{AC}([0,T],\overline{\Omega})},\x_{\eps}(T)=x\right\} .
\]

 We define the compact slow-manifold for the Hamilton-Jacobi equation by
\begin{equation}\label{eq:MS}
    \MS:=\overline{\Omega} \cap \cE_\fast.
\end{equation}
  
Combining the $\Gamma$-convergence result in Theorem \ref{thm:GammaConvergence} and the uniform convergence from Corollary \ref{cor:ExistenceLimitHJE}, we give a representation formula for the limit $u^*$ in the following proposition.

\begin{prop}\label{prop:identify}
Given any $T>0$, assume the initial data $u^\eps_0$ satisfies Assumption \ref{assumption} and $u^\eps$ is the viscosity solution to \eqref{eq:FastSlowHJE}. Assume both   \eqref{FDB} and   \eqref{eq:UFEC} conditions.
Let $\mathcal{L}_\eff$ be defined by \eqref{eq:EffectiveAF}, and let $\hat{u}$ be the limit of $u^\eps$ in Corollary \ref{cor:ExistenceLimitHJE}.    {   Then   for all $x\in\MS$,  we have the unique variational representation } 
\begin{align*}
u^{*}(x,T):=\lim_{\eps\to 0} u^\eps(x,T)=&\inf\left\{ u^*_{0}(\x(0))+{\cal L}_{\eff}(\x):\x\in\mathrm{AC}([0,T],\overline{\Omega}),\,\x(T)=x\right\} \\
=&\inf\left\{ u^*_{0}(\x(0))+{\cal L}_{\eff}(\x): \x\in\mathrm{AC}([0,T],\MS),\,\x(T)=x\right\} .
\end{align*}
{   Moreover, $u^\eps$  converges to $u^*=\hat{u}$ uniformly on $K\times[0,T]$ for  any compact subset $K\subset \MS^o.$}
\end{prop}

\begin{proof}
Let us fix $(x,T)\in\MS\times(0,+\8)$. To simplify notations, we define that $X:= \{ \x\in\mathrm{AC}([0,T],\overline{\Omega}),\,\,\x(T)=x \} $
and
\[
F_{\eps}(\x):=u_{0}^\eps(\x(0))+{\cal L}_{\eps}(\x),\quad F_{0}(\x):=u_{0}^*(\x(0))+{\cal L}_{\eff}(\x).
\]
    Let $\x_{\eps}\in \mathrm{AC}([0,T],\overline{\Omega}) \subset \mathrm{AC}([0,T],\sC)$    such that $F_\eps(\x_{\eps})=\inf F_{\eps}$. Then, we
have that $F_{\eps}(\x_{\eps})<\infty$. Since, $u_{0}^\eps$ is uniformly bounded and $\x_{\eps}(0)\in\overline{\Omega}$ we conclude that ${\cal L}_{\eps}(\x_{\eps})<\infty$.
By the   compactness results in Proposition \ref{prop:Compactness}, we conclude that there is a limit $\x_{0}$ such that $\x_{\eps}\to \x_{0}$ in $C([0,T],\overline{\Omega})$, and moreover, $\x_0\in \text{\ensuremath{\mathrm{AC}}}_x([0,T],\MS)$. We want to show that $F_{0}(\x_{0})=\inf F_{0}$.

Indeed, we have
\begin{equation}\label{tm:F1}
\begin{aligned}
\inf F_{0}\leq F_{0}(\x_{0})=u_{0}^*(\x_{0}(0))+{\cal L}_{\eff}(\x_{0})\leq&\lim_{\eps\to0}u_{0}^\eps(\x_{\eps}(0))+\liminf_{\eps\to0}{\cal L}_{\eps}(\x_{\eps})\\
\leq&\liminf_{\eps\to0}F_{\eps}(\x_{\eps})=\liminf_{\eps\to0}\inf F_{\eps}=\liminf_{\eps\to0} u^\eps,
\end{aligned}
 \end{equation}
where we have used the $\Gamma$-convergence of ${\cal L}_{\eps}\to{\cal L}_{\eff}$,
the equicontinuity of $u^\eps_{0}$, and the convergence $\x_{\eps}\to \x_{0}\text{ in } C([0,T],\overline{\Omega})$.
Moreover by the construction of the recovery sequence (see \ref{prop:ConstructionRecoverySequence}), for any $\x\in X$ there is sequence $\bar{\x}_{\eps}\to \x$
in $X$ (in particular $\x_\eps(0)\to \x(0)$) such that ${\cal L}_{\eps}(\bar{\x}_{\eps})\to{\cal L}_{\eff}(\x)$.
Hence, we have that 
\begin{align*}
F_{0}(\x) & =u^*_{0}(\x(0))+{\cal L}_{\eff}(\x)\geq\lim_{\eps\to0}u^\eps_{0}(\bar{\x}_{\eps}(0))+\limsup_{\eps\to0}{\cal L}_{\eps}(\bar{\x}_{\eps})\\
 & \geq\limsup_{\eps\to0}u_{0}^\eps(\bar{\x}_{\eps}(0))+{\cal L}_{\eps}(\bar{\x}_{\eps})=\limsup_{\eps\to0}F_{\eps}(\bar{\x}_{\eps})\geq\limsup_{\eps\to0}\inf F_{\eps}=\limsup_{\eps\to0} u^\eps.
\end{align*}
Taking the infimum w.r.t. $\x\in X$ we conclude that $\inf F_{0}\geq\limsup_{\eps\to0}\inf F_{\eps}$. This. together with \eqref{tm:F1},
  hence implies that 
\[
\lim_{\eps\to0} u^\eps =\lim_{\eps\to0}\inf F_{\eps}=\inf F_{0},\quad\text{and }F_0(\x_{0})=\inf F_{0}.
\]
This means that for $x\in\mathcal{M}_\S\subset \overline{\Omega}$ we have 
$$u^{*}(x,T):=\lim_{\eps\to0}u^{\eps}(x,T) =\inf F_{0}=\inf\left\{ u^*_{0}(\x(0))+{\cal L}_{\eff}(\x): \x\in\mathrm{AC}([0,T],\MS),\,\x(T)=x\right\}.$$
{   Moreover, $u^*(x,t)=\hat{u}(x,t)$ for any interior point $x\in \MS^o$. 
  Thus the subsequence  convergence for $u^\eps$ obtained in Corollary \ref{cor:ExistenceLimitHJE} is actually a uniform convergence for all $\eps \to 0$ since the limit $\lim_{\eps\to 0} u^\eps(x,T)=u^{*}(x,T)=\inf F_{0}, x\in \MS$ is unique.
 }
\end{proof}
The above characterization provides that the limit $u^*$ can be expressed on the compact manifold $\MS$ (possibly with boundary) by the least-action representation
\begin{equation}\label{eq:RepresentationUstar}
u^*(x,t)=\inf_{ \mathrm{x}\in \mathrm{AC}_x([0,t]; \MS)} \big\{u_0^*(\mathrm{x}(0)) +  
   \int_0^t L_\eff(\mathrm{x}_s, \dot{\mathrm{x}}_s)\ud s\big\}. 
\end{equation}

\begin{rem}
    The above identification provides a formula for $u^*=\hat{u}$ on the slow manifold $ \MS$. A natural question is, how the limit looks like away from $\MS$, because the variational representation is not applicable. Corollary \ref{cor:ExistenceLimitHJE} only shows that $\hat{u}$ is continuous and Lipschitz, however a more detailed characterization is open and left for further analysis.
\end{rem}

\subsection{
Viscosity solution of effective HJE with $H_\eff$
}
Recall the compact slow manifold $\MS$ defined in \eqref{eq:MS}.
We are going to verify by the dynamic programming that $u^*$ from \eqref{eq:RepresentationUstar} satisfied the effective state-constraint HJE \eqref{HJE0} on the manifold $\MS$, i.e., 
\begin{equation*} 
\begin{aligned}
\pt_t u(x,t) + H_\eff(x,d_xu(x,t))\leq 0, \quad (x,t)\in \MS^o\times(0,+\8),\\
\pt_t u(x,t) + H_\eff(x,d_x u(x,t))\geq 0, \quad (x,t)\in \MS\times(0,+\8),\\
u(x,0) = u_0(x), \quad x\in \MS=\overline{\MS}, 
\end{aligned}
\end{equation*}
where $\MS^o$ denotes the interior of $\MS$.
Here, 
 recall that the differential $d_xu$ in differential geometry of a smooth function $u(\cdot, t):\MS\to\R$ at the point $x_0\in\MS$ is defined by 
$$d_x u(x_0,t)[v]:= \lim_{\tau \to 0}\frac{u(\x(\tau), t)-u(\x(0), t)}{\tau},$$
where $\x:\R\to\MS$ is a smooth curve with $\x(0)=x_0, \, \dot{\x}(0)=v\in T_x\MS$. (Note that the differential acts only in the first argument $x\in\MS$.)
Viscosity solutions of \eqref{HJE0} on manifolds  $\MS\times(0,+\8)$ are understood as follows, which we recall from 
\cite{capuzzo1988hamilton,fathi2012weak}:
\begin{defn}\label{def:visManifold}
    A function $u$ is a viscosity solution to \eqref{HJE0} if it is both a viscosity subsolution and   supersolution in the following sense: $u$ is a viscosity subsolution of 
  $$\pt_t u +  H_\eff(x,d_xu) \leq 0$$
  on  $\MS^o\times(0,+\8)$, if for every $\phi \in C^1(\MS\times(0,+\8))$   and every $(x_0,t_0)\in \MS^o\times(0,+\8)$ such that $u-\phi$ has a maximum at $(x_0, t_0)$, we have
  $$\pt_t \phi(x_0, t_0) + H_\eff(x_0, d_{x}\phi(x_0,t_0))\leq 0;$$
    $u$ is a viscosity supersolution of 
  $$\pt_t u +  H_\eff(x,d_xu) \geq 0$$
  on  $\overline{\MS}\times(0,+\8)$, if for every $\phi \in C^1(\overline{\MS}\times(0,+\8))$  and every $(x_0,t_0)\in \overline{\MS}\times(0,+\8)$ such that $u-\phi$ has a minimum at $(x_0, t_0)$, we have
  $$\pt_t \phi(x_0, t_0) + H_\eff(x_0, d_{x}\phi(x_0,t_0)))\geq 0.$$
\end{defn}
\begin{thm}\label{thm:effectiveHJE}
 The limiting function   $u^*$ represented as  \eqref{eq:RepresentationUstar} is a viscosity solution to the effective HJE \eqref{HJE0} on the manifold $\MS$ in the sense of   Definition \ref{def:visManifold}.
\end{thm}
\begin{proof}
We verify both conditions separetely.\\
\textbf{Step 1:} $u^*$ is a subsolution:\\
 Let $\phi\in C^1(\MS\times(0,+\8))$ such that $u-\phi$ attains its maximum at $(x_0, t_0)\in \MS^o\times (0,+\8)$. Then for any $ \mathrm{x}\in \mathrm{AC}_{x_0}([0,t_0]; \MS)$ with $\mathrm{x}(t_0)=x_0$, we have
    $$\phi(x_0, t_0)\leq \int_\tau^{t_0} L_\eff(\mathrm{x}_s, \dot{\mathrm{x}}_s)\ud s + u(\mathrm{x}(\tau),\tau)\leq \int_\tau^{t_0} L_\eff(\mathrm{x}_s, \dot{\mathrm{x}}_s)\ud s + \phi(\mathrm{x}(\tau),\tau).$$
  Rearranging and dividing by $t_0-\tau$, we have for any $v\in T_{x_0}\MS$,
  \begin{equation}
      \frac{\phi(x_0, t_0)-\phi(\mathrm{x}(\tau),\tau)}{t_0-\tau}\leq \frac{1}{t_0-\tau} \int_\tau^{t_0} L_\eff(\x^v_s, \dot{\x}^v_s )\ud s,
  \end{equation}
  where $\x^v_s$ is any curve on $\MS$ such that $\x^v_{s=t_0}=x_0, \dot{\x}^v_{s=t_0}=v.$
    Taking $\tau\to t_0$, we have
  \begin{equation}
      \pt_t\phi(x_0, t_0) + d_{x}\phi(x_0,t_0)[v] - L_\eff(x_0,v)\leq 0, \quad \forall v\in T_{x_0}\MS.
  \end{equation}
   Taking supremum with respect to $v\in T_{x_0}\MS$ and using Lemma \ref{lem:HeffRepresentation}, this proves 
  $$\pt_t\phi(x_0, t_0) + H_\eff(x_0, d_{x}\phi(x_0 , t_0)) \leq 0.$$
\\
\textbf{Step 2:} $u^*$ is a supersolution:\\
Let $\phi\in C^1(\overline{\MS}\times(0,+\8))$ such that $u-\phi$ attains its minimum at $(x_0, t_0)\in \overline{\MS}\times (0,+\8)$.
  Then we have
  \begin{align*}
      \phi(x_0, t_0)= \inf\{  \int_\tau^{t_0} L_\eff(\mathrm{x}_s, \dot{\mathrm{x}}_s)\ud s + u(\mathrm{x}(\tau),\tau) \}\geq \inf\{\int_\tau^{t_0} L_\eff(\mathrm{x}_s, \dot{\mathrm{x}}_s)\ud s + \phi(\mathrm{x}(\tau),\tau)\},
  \end{align*}
  where  $\mathrm{x}(\cdot)\in \mathrm{AC}_{x_0}([0,t_0]; \overline{\MS})$. 
  Thus 
  \begin{align*}
      0\leq \sup\{\phi(x_0, t_0) - \phi(\mathrm{x}(\tau),\tau) - \int_\tau^{t_0} L_\eff(\mathrm{x}_s, \dot{\mathrm{x}}_s)\ud s\}.
  \end{align*}
  Hence
  \begin{align*}
      0\leq \frac{1}{t_0-\tau}\sup\{\int_\tau^{t_0} (\pt_t \phi + d_{x} \phi[\dot{\mathrm{x}}_s])] \ud s - \int_\tau^{t_0} L_\eff(\mathrm{x}_s, \dot{\mathrm{x}}_s)\ud s\}.
  \end{align*}
  By the definition of supremum,
  for any $\eps$, there exists $\mathrm{x}$ such that 
  \begin{align*}   
  -\eps\leq \frac{1}{t_0-\tau} \int_\tau^{t_0} [ \pt_t \phi + d_x\phi(\x_s)[\dot{\mathrm{x}}_s]    -   L_\eff(\mathrm{x}_s, \dot{\mathrm{x}}_s)]\ud s \leq \frac{1}{t_0-\tau} \int_\tau^{t_0} [ \pt_t \phi + H_\eff(\mathrm{x}_s,d_x\phi(\x_s)))]\ud s.
  \end{align*}
Taking $\tau \to t_0$, we have
$$ \pt_t \phi(x_0, t_0)+ H_\eff(x_0, d_{x}\phi(x_0,t_0))\geq 0.$$
Thus, we conclude $u^*$ is a viscosity solution to \eqref{HJE0}.
\end{proof}

\subsection{Coarse-grained HJE} We have seen that the limit functionals $L_\eff$ and $H_\eff$ can also be expressed in terms of coarse-grained variables, which in turn define coarse-grained functionals $\sL$ and $\sH$, see \eqref{s:CGLagrangian} and \eqref{eq:CGHamiltonian}, respectively. We are going to show that also the limit $u^*$ of solutions of \eqref{eq:FastSlowHJE}  as a function on $\MS$ can be expressed in coarse-grained variables, which then solves the coarse-grained HJE. Recall the reconstruction map $\sR$ from \eqref{eq:UFEC}.
Using the explicit parametrization of the slow-manifold $\MS$, we define a new continuous function 
$$\su:\sQ\to\R,\quad \su(\sq):=u^*(\sR(\sq)).$$
\begin{prop}
    The function $\su=u^*\circ \sR$ is a viscosity solution of the HJE with the Hamiltonian $\sH$ in \eqref{eq:CGHamiltonian} on the (flat) manifold $\sQ$ in the sense of Definition \ref{def:vis}.
\end{prop}
\begin{proof}
    We only show how the derivatives of $u^*$ translate by the chain-rule. We have $\partial_t\su(\sq,t)=\partial_t u^*(\sR(\sq),t)$. Moreover,  we know that $H_\eff(\sR(\sq),p)=\sH(\sq,\sp)$ if $Q_\fast^T\sp=p$.
    Hence, it suffices to show that $d_{\sR(\sq)}u^* = Q^T_\fast d_\sq\su$, or, equivalently, $d_{\sR(\sq)}u^*[v] = Q^T_\fast d_\sq\su[v]$ for all $v\in T_{\sR(\sq)}\MS$. Using the projection onto the tangent space $T_x\MS$ and its explicit characterization, we have  $v=d_\sq\sR\circ Q_\fast[v]$ (see Proposition \ref{prop:PropertiesReconstructionAndProjection}). Hence, by the chain rule
    $$
    d_{\sR(\sq)}u^*[v] = d_{\sR(\sq)}
    u^*
    \circ d_\sq\sR\circ Q_\fast[v]=d_\sq\su\circ Q_\fast[v] \quad \Rightarrow \quad d_{\sR(\sq)}u^* = Q^T_\fast d_\sq\su.
    $$
\end{proof}
We finally remark that again $\su$ has an integral representation. To see this, we define:
\begin{equation*}
\widetilde{\su}(\sq,t)=\inf_{ \mathrm{y}\in \mathrm{AC}([0,t]; \sQ),\, \mathrm{y}(t)=\sq} \big(u_0(\sR(\mathrm{y}(0))) +  
   \int_0^t \sL(\mathrm{y}_s, \dot{\mathrm{y}}_s)\ud s\big). 
\end{equation*}
Using the same procedure as in Theorem \ref{thm:effectiveHJE}, we can prove that $\widetilde{\su}(\sq,t)$ is a viscosity solution of HJE with the Hamiltonian $\sH$.
By uniqueness of viscosity solution, we obtain the representation formula for $\su(\sq,t)=\widetilde{\su}(\sq,t)$.

\section*{Acknowledgment}
Yuan Gao was partially supported by NSF under award DMS-2204288 and CAREER award DMS-2440651. We would also like to thank the anonymous referees for their constructive comments.

\bibliographystyle{alpha}
\bibliography{HJEbib}

\newcommand{\etalchar}[1]{$^{#1}$}
\begin{thebibliography}{KMK73}

\bibitem[ABM07]{alvarez2007multiscale}
Olivier Alvarez, Martino Bardi, and Claudio Marchi.
\newblock Multiscale problems and homogenization for second-order
  hamilton--jacobi equations.
\newblock {\em Journal of Differential Equations}, 243(2):349--387, 2007.

\bibitem[AK15]{Kurtz15}
David~F. Anderson and Thomas~G. Kurtz.
\newblock {\em Stochastic Analysis of Biochemical Systems}.
\newblock Springer International Publishing, 2015.

\bibitem[BCC10]{buice2010systematic}
Michael~A Buice, Jack~D Cowan, and Carson~C Chow.
\newblock Systematic fluctuation expansion for neural network activity
  equations.
\newblock {\em Neural computation}, 22(2):377--426, 2010.

\bibitem[BD{\etalchar{+}}97]{bardi1997optimal}
Martino Bardi, Italo~Capuzzo Dolcetta, et~al.
\newblock {\em Optimal control and viscosity solutions of
  Hamilton-Jacobi-Bellman equations}, volume~12.
\newblock Springer, 1997.

\bibitem[BN13]{bressloff2013metastability}
Paul~C Bressloff and Jay~M Newby.
\newblock Metastability in a stochastic neural network modeled as a velocity
  jump markov process.
\newblock {\em SIAM Journal on Applied Dynamical Systems}, 12(3):1394--1435,
  2013.

\bibitem[Bot03]{bothe2003reaction}
Dieter Bothe.
\newblock Instantaneous limits of reversible chemical reactions in presence of
  macroscopic convection.
\newblock {\em Journal of Differential Equations}, 193(1):27--48, 2003.

\bibitem[CD88]{capuzzo1988hamilton}
Italo Capuzzo-Dolcetta.
\newblock Hamilton-{J}acobi equations with constraints.
\newblock In {\em Stochastic Differential Systems, Stochastic Control Theory
  and Applications}, pages 99--106. Springer, 1988.

\bibitem[Fat12]{fathi2012weak}
Albert Fathi.
\newblock Weak kam from a pde point of view: viscosity solutions of the
  {H}amilton--jacobi equation and {A}ubry set.
\newblock {\em Proceedings of the Royal Society of Edinburgh Section A:
  Mathematics}, 142(6):1193--1236, 2012.

\bibitem[FL07]{fonleo07}
Irene Fonseca and Giovanni Leoni.
\newblock Modern methods in the calculus of variations: $l^p$ spaces, 2007.
\newblock cvgmt preprint.

\bibitem[Gan87]{Hu87}
Hu~Gang.
\newblock Stationary solution of master equations in the large-system-size
  limit.
\newblock {\em Physical Review A}, 36(12):5782–5790, Dec 1987.

\bibitem[GL22]{GL22}
Yuan Gao and Jian-Guo Liu.
\newblock Revisit of macroscopic dynamics for some non-equilibrium chemical
  reactions from a {H}amiltonian viewpoint.
\newblock {\em Journal of Statistical Physics}, 189(2):1--57, 2022.

\bibitem[GL23]{GL23}
Yuan Gao and Jian-Guo Liu.
\newblock Large deviation principle and thermodynamic limit of chemical master
  equation via nonlinear semigroup.
\newblock {\em Multiscale Modeling \& Simulation}, 21(4):1534--1569, 2023.

\bibitem[GQ17]{QianGe17}
Hao Ge and Hong Qian.
\newblock Mathematical formalism of nonequilibrium thermodynamics for nonlinear
  chemical reaction systems with general rate law.
\newblock {\em Journal of Statistical Physics}, 166(1):190--209, 2017.

\bibitem[GQX15]{PhysRevLett}
Hao Ge, Hong Qian, and X.~Sunney Xie.
\newblock Stochastic phenotype transition of a single cell in an intermediate
  region of gene state switching.
\newblock {\em Phys. Rev. Lett.}, 114:078101, Feb 2015.

\bibitem[KK13]{kang2013separation}
Hye-Won Kang and Thomas~G Kurtz.
\newblock Separation of time-scales and model reduction for stochastic reaction
  networks.
\newblock {\em The Annals of Applied Probability}, 23(2):529--583, 2013.

\bibitem[KMK73]{Kubo73}
Ryogo Kubo, Kazuhiro Matsuo, and Kazuo Kitahara.
\newblock Fluctuation and relaxation of macrovariables.
\newblock {\em Journal of Statistical Physics}, 9(1):51–96, Sep 1973.

\bibitem[Kur80]{kurtz1980representations}
Thomas~G Kurtz.
\newblock Representations of {M}arkov processes as multiparameter time changes.
\newblock {\em The Annals of Probability}, pages 682--715, 1980.

\bibitem[LL17]{li2017large}
Tiejun Li and Feng Lin.
\newblock Large deviations for two-scale chemical kinetic processes.
\newblock {\em Communications in Mathematical Sciences}, 15(1):123--163, 2017.

\bibitem[LX11]{li2011central}
Gene-Wei Li and X~Sunney Xie.
\newblock Central dogma at the single-molecule level in living cells.
\newblock {\em Nature}, 475(7356):308--315, 2011.

\bibitem[Mit08]{mitake2008asymptotic}
Hiroyoshi Mitake.
\newblock Asymptotic solutions of {H}amilton-{J}acobi equations with state
  constraints.
\newblock {\em Applied Mathematics and Optimization}, 58:393--410, 2008.

\bibitem[MPS21]{MielkePeletierStephan}
Alexander Mielke, Mark~A Peletier, and Artur Stephan.
\newblock {EDP}-convergence for nonlinear fast–slow reaction systems with
  detailed balance.
\newblock {\em Nonlinearity}, 34(8):5762, jul 2021.

\bibitem[MS20]{MielkeStephan}
Alexander Mielke and Artur Stephan.
\newblock Coarse-graining via {EDP}-convergence for linear fast-slow reaction
  systems.
\newblock {\em Mathematical Models and Methods in Applied Sciences},
  30(09):1765--1807, 2020.

\bibitem[PR23]{PeletierRenger}
Mark~A. Peletier and D.~R.~Michiel Renger.
\newblock Fast reaction limits via {$\Gamma$}-convergence of the flux rate
  functional.
\newblock {\em Journal of Dynamics and Differential Equations}, 35, 2023.

\bibitem[QG21]{QianBook}
Hong Qian and Hao Ge.
\newblock {\em Stochastic Chemical Reaction Systems in Biology}.
\newblock Lecture Notes on Mathematical Modelling in the Life Sciences.
  Springer International Publishing, 2021.

\bibitem[Ste21]{Stephan21}
Artur Stephan.
\newblock {EDP}-convergence for a linear reaction-diffusion system with fast
  reversible reaction.
\newblock {\em Calculus of Variations and Partial Differential Equations},
  60(6):226, 2021.

\bibitem[Tra21]{Tran21}
Hung~Vinh Tran.
\newblock {\em Hamilton-Jacobi equations: theory and applications}, volume 213.
\newblock American Mathematical Soc., 2021.

\end{thebibliography}

\end{document}